%% file: Spectral_invariants_of_Joyce_orbifolds.tex
\documentclass[10pt,reqno]{amsart}
\pdfoutput=1
\input{Festino_1pt0}

\title{Spectral invariants of Joyce orbifolds}
\author{Laurence H. Mayther}

\begin{document}\fontsize{10pt}{12pt}\selectfont
\begin{abstract}
\footnotesize{This paper introduces two new spectral invariants of torsion-free \g-structures on closed orbifolds and computes their values on all Joyce orbifolds.  These invariants are shown to be more discerning than the \(\ol{\nu}\)-invariant of \CGN\ when applied to Joyce orbifolds, and thus provide candidate tools for distinguishing between Joyce manifolds.  The invariants may be viewed as regularisations of the classical Morse indices of the critical points of the Hitchin functionals on closed and coclosed \g-structures respectively.  In the case of Joyce orbifolds, an interesting link with twisted Epstein \(\ze\)-functions is also observed.}
\end{abstract}
\maketitle

\section{Introduction}

Let \(\M\) be an oriented 7-manifold.  A 3-form \(\ph\) on \(\M\) is termed a \g\ 3-form if, for each \(x \in \M\), there exists an orientation-preserving isomorphism \(\al:\T_x\M \to \bb{R}^7\) identifying \(\ph|_x\) with the 3-form:
\ew
\vph_0 = \th^{123} + \th^{145} + \th^{167} + \th^{246} - \th^{257} - \th^{347} - \th^{356} \in \ww{3}\lt(\bb{R}^7\rt)^*,
\eew
where \(\lt(\th^1,...,\th^7\rt)\) denotes the standard basis of \(\lt(\bb{R}^7\rt)^*\) and multi-index notation \(\th^{ij...k} = \th^i \w \th^j \w ... \w \th^k\) is used throughout this paper.  Since the stabiliser of \(\vph_0\) in \(\GL_+(7;\bb{R})\) is \g, a \g\ 3-form \(\ph\) on \(\M\) is equivalent to a \g-structure on \(\M\), i.e.\ a principal \g-subbundle of the frame bundle of \(\M\).  Since \(\Gg_2 \pc \SO(7)\), such a \(\ph\) induces a metric \(g_\ph\) and volume form \(vol_\ph\) on \(\M\).  In particular, \(\ph\) defines a 4-form \(\ps = \Hs_\ph\ph\), where \(\Hs_\ph\) denotes the Hodge star induced by \(\ph\); call any 4-form which arises in this way a \g\ 4-form.  Say that \(\ph\) and \(\ps\) (or, more generally, the corresponding \g-structure on \(\M\)) are torsion-free if both \(\dd\ph = 0\) and \(\dd\ps = 0\); in this case, the holonomy group of the metric \(g_\ph\) is a subgroup of the exceptional holonomy group \g.  More generally, if \(\M\) is an orbifold, say that \(\ph \in \Om^3(\M)\) is a \g\ 3-form if for each chart \(U \cong \lqt{\bb{R}^7}{\Ga}\) in \(\M\), \(\ph|_U\) may be lifted to a \(\Ga\)-invariant \g\ 3-form on \(\bb{R}^7\); define \g\ 4-forms analogously and again say that \(\ph\) and \(\ps\) are torsion-free if both \(\dd\ph = 0\) and \(\dd\ps = 0\).

Now suppose that \(\M\) is closed and let \(\ph\) be a closed \g\ 3-form on \(\M\).  Write \([\ph] \in \dR{3}(\M)\) for the de Rham class defined by \(\ph\) and define \([\ph]_+\) to be \(\lt\{ \ph' \in [\ph] ~\m|~ \ph' \text{ is a \g\ 3-form}\rt\}\).  The Hitchin functional \(\CH: [\ph]_+ \to (0,\0)\) is defined by:
\ew
\CH(\ph') = \bigintsss_\M vol_{\ph'}.
\eew
Then \(\ph' \in [\ph]_+\) is a critical point of \(\CH\) \iff\ it is torsion-free.  Likewise, given a closed \g\ 4-form \(\ps\) on \(\M\), define \([\ps]_+\) to be \(\lt\{ \ps' \in [\ps] ~\m|~ \ps' \text{ is a \g\ 4-form}\rt\}\) and define the Hitchin functional \(\CCH: [\ps]_+ \to (0,\0)\) by:
\ew
\CCH(\ps') = \bigintsss_\M vol_{\ps'}.
\eew
Once again, \(\ps' \in [\ps]_+\) is a critical point of \(\CCH\) \iff\ it is torsion-free.

As proven in \cite{TGo3Fi6&7D}, the critical points of the functional \(\CH\) are non-degenerate local maxima, modulo the actions of diffeomorphisms.  Likewise, \pref{CCH-Prop} proves that the critical points of \(\CCH\) are non-degenerate saddles, modulo the action of diffeomorphisms.\footnote{After proving \pref{CCH-Prop}, the author discovered that a related result was obtained in \cite{STBoaMLCoG2S}.  Note, however, that \pref{CCH-Prop} differs from \cite[Prop.\ 3.4]{STBoaMLCoG2S}, since it not only proves that the critical points of \(\CCH\) are non-degenerate, it also proves that they are saddles.}  Motivated by classical Morse theory, it is natural to ask, therefore, whether the critical points of \(\CH\) and \(\CCH\) have well-defined Morse indices.  Whilst the classical Morse indices of the critical points are not well-defined, in \sref{SMI}, by building on \cite{OeI}, I introduce a new, spectral, `regularised' notion of Morse indices for functionals, which I term the spectral Morse index.

\begin{Thm}[See \trefs{mu3-def-thm} and \ref{mu4-def-thm}]
Let \(\M\) be a closed, oriented 7-orbifold and let \(\ph\) be a torsion-free \g\ 3-form on \(\M\), with corresponding \g\ 4-form \(\ps\).  Then, viewing \(\ph\) as a critical point of the functional \(\CH: [\ph]_+ \to (0,\0)\), the spectral Morse index of \(\ph\) is well-defined.  Likewise, the spectral Morse index of \(\ps\) is well-defined, viewed as a critical point of the functional \(\CCH: [\ps]_+ \to (0,\0)\).  Denote the indices by \(\mu_3(\ph)\) and \(\mu_4(\ps)\), respectively, and let \(\cal{G}^{TF}_2(\M)\) denote the moduli space of torsion-free \g-structures on \(\M\) (i.e.\ the quotient of the space of torsion-free \g\ 3-forms on \(\M\) by the action of \(\Diff_0(\M)\)).  Then \(\mu_3\) and \(\mu_4\) define smooth maps:
\ew
\mu_3,\mu_4: \cal{G}^{TF}_2(\M) \to \bb{R},
\eew
which are invariant under rescaling \(\ph \mt \ell^3\ph\) (\(\ell > 0\)).
\end{Thm}

Thus, the spectral Morse indices \(\mu_3\) and \(\mu_4\) define two new invariants of torsion-free \g-structures on orbifolds.  The second half of this paper is devoted to computing these invariants explicitly on Joyce orbifolds.  Let \(\bb{T}^7 = \rqt{\bb{R}^7}{\bb{Z}^7}\), let \(\Ga \pc \SL(7;\bb{Z}) \sdp \bb{T}^7\) be a finite subgroup and consider the quotient \(\M_\Ga = \lqt{\bb{T}^7}{\Ga}\).  \(\M_\Ga\) is called a Joyce orbifold if it admits torsion-free \g-structures.  Given \(A \in \End(\bb{R}^7)\), define:
\ew
\Tr_8^{\SU(3)}(A) = \frac{\Tr(A)^2 - \Tr(A^2)}{2} - 2\Tr(A) + 1
\eew
and:
\ew
\Tr_{12}^{\SU(3)}(A) = \frac{\Tr(A)^3 + 2\Tr(A^3) - 3\Tr(A^2)\Tr(A)}{6} - \frac{\Tr(A)^2 - \Tr(A^2)}{2} - 2.
\eew

\begin{Thm}[See \trefs{mu3-J-thm} and \ref{mu4-J-thm}]\label{Joyce-comp-thm}
Let \(\M_\Ga = \lqt{\bb{T}^7}{\Ga}\) be a Joyce orbifold.  The \(\mu_3\)- and \(\mu_4\)-invariants \(\mu_3, \mu_4:\cal{G}^{TF}_2(\M_\Ga) \to \bb{R}\) are constant, taking the values:
\ew
\mu_3(\M_\Ga) = \frac{-1}{|\Ga|} \sum_{\mc{A} = (A,t) \in \Ga} \Tr^{\SU(3)}_8(A)
\eew
and:
\ew
\mu_4(\M_\Ga) = \frac{-1}{|\Ga|} \sum_{\mc{A} = (A,t) \in \Ga} \Tr^{\SU(3)}_{12}(A),
\eew
respectively.
\end{Thm}

The proof of \tref{Joyce-comp-thm} reveals an interesting link between the \(\mu\)-invariants and twisted Epstein \(\ze\)-functions, as introduced in \cite{ZTAZI}; see \sref{mu-comp-RepT} for details.  I also remark that, despite much recent progress \cite{AAIoG2M, nIoETCS} in obtaining computable invariants of manifolds with torsion-free \g-structures constructed using the `twisted connected sum' construction described in \cite{TCS&SRH,ACCY3FfWF3F,ETCSG2M}, obtaining computable invariants of torsion-free \g-structures on Joyce manifolds (i.e.\ manifolds obtained  by resolving the singularities in Joyce orbifolds as in \cite{CR7MwHG2I, CR7MwHG2II, CMwSH}) is still an open problem.  Due to the computability of \(\mu_3\) and \(\mu_4\) on Joyce orbifolds, it is the author's hope that these invariants will provide computable invariants on Joyce manifolds.

The results of this paper were obtained during the author's doctoral studies, which where supported by EPSRC Studentship 2261110.\\

\section{Preliminaries}

\subsection{Orbifolds}\label{Orb}

The material in this subsection is largely based on \cite[\S\S1.1--1.3]{O&ST} and \cite[\S14.1]{THKLFPFftScDO}.  Let \(E\) be a topological space.  
\begin{Defn}
An \(n\)-dimensional orbifold chart \(\Xi\) is the data of a connected, open neighbourhood \(U\) in \(E\), a finite subgroup \(\Ga \pc \GL(n;\bb{R})\), a connected, \(\Ga\)-invariant open neighbourhood \(\tld{U}\) of \(0 \in \bb{R}^n\) and a homeomorphism \(\ch: \lqt{\tld{U}}{\Ga} \to U\).  Write \(\tld{\ch}\) for the composite \(\tld{U} \oto{quot} \lqt{\tld{U}}{\Ga} \oto{\ch} U\).  Say that \(\Xi\) is centred at \(e \in E\) if \(e = \tld{\ch}(0)\).   In this case, \(\Ga\) is called the orbifold group of \(e\), denoted \(\Ga_e\).

Now consider two orbifold charts \(\Xi_1 = \lt(U_1,\Ga_1,\tld{U}_1,\ch_1\rt)\) and \(\Xi_2 = \lt(U_2,\Ga_2,\tld{U}_2,\ch_2\rt)\) with \(U_1 \cc U_2\).   An embedding of \(\Xi_1\) into \(\Xi_2\) is the data of a smooth, open embedding \(\io_{12}:\tld{U}_1 \emb \tld{U}_2\) and a group isomorphism \(\la_{12}: \Ga_1 \to \Stab_{\Ga_2}(\io_{12}(0))\) such that for all \(x \in \tld{U}_1\) and all \(\si \in \Ga_1\), \(\io_{12}(\si \cdot x) = \la_{12}(\si) \cdot \io_{12}(x)\), and such that the following diagram commutes:
\ew
\bcd[column sep = 13mm, row sep = 3mm]
\tld{U}_1 \ar[rr, ^^22 \mla{\io_{12}} ^^22] \ar[dd, ^^22 \mla{\tld{\ch}_1} ^^22] & & \tld{U}_2 \ar[dd, ^^22 \mla{\tld{\ch}_2} ^^22]\\
& & \\
U_1 \ar[rr, ^^22\mla{incl}^^22, hook] &  & U_2
\ecd
\eew

Now let \(\Xi_1\) and \(\Xi_2\) be arbitrary.  \(\Xi_1\) and \(\Xi_2\) are compatible if for every \(e \in U_1 \cap U_2\), there exists a chart \(\Xi_e = \lt(U_e,\Ga_e,\tld{U}_e,\ch_e\rt)\) centred at \(e\) together with embeddings \((\io_{e1},\la_{e1}): \Xi_e \emb \Xi_1\) and \((\io_{e2},\la_{e2}): \Xi_e \emb \Xi_2\).  If \(U_1 \cap U_2 = \es\), then \(\Xi_1\) and \(\Xi_2\) are automatically compatible, however if \(U_1 \cap U_2 \ne \es\) and \(\Xi_1\) and \(\Xi_2\) are compatible, then \(\Xi_1\) and \(\Xi_2\) have the same dimension; moreover, if \(\Xi_1\) and \(\Xi_2\) are centred at the same point \(e \in E\), then \(\Ga_1 \cong \Ga_2\) and thus the orbifold group \(\Ga_e\) is well-defined up to isomorphism.

An orbifold atlas for \(E\) is a collection of compatible orbifold charts \(\fr{A}\) which is maximal in the sense that if a chart \(\Xi\) is compatible with every chart in \(\fr{A}\), then \(\Xi \in \fr{A}\).  An orbifold is a connected, Hausdorff, second-countable topological space \(E\) equipped with an orbifold atlas \(\fr{A}\).  Every chart of \(E\) has the same dimension \(n\); call this the dimension of the orbifold.
\end{Defn}

Let \(E\) be an \(n\)-orbifold.  Given any chart \(\Xi = \lt(U, \Ga, \tld{U}, \ch\rt)\) for \(E\), the action of \(\Ga\) on \(\tld{U}\) naturally lifts to an action of \(\Ga\) on \(\T\tld{U}\) by bundle automorphisms.  Given a second chart \(\Xi' = (U',\Ga',\tld{U}',\ch')\) embedding into \(\Xi\), the map \(\tld{U}' \emb \tld{U}\) induces an equivariant embedding of bundles \(\T\tld{U}' \emb \T\tld{U}\).  Define:
\ew
\T E = \rqt{\lt[\coprod_{\Xi} \lt(\lqt{\T\tld{U}}{\Ga}\rt)\rt]}{\tl}
\eew
where the quotient by \(\tl\) denotes that one should glue along all embeddings \(\T\tld{U}' \emb \T\tld{U}\).  The resulting space \(\T E\) is an orbifold, and the natural projection \(\pi: \T E \to E\) gives \(\T E\) the structure of an orbifold vector bundle over \(E\) in the sense of \cite[p.\ 173]{THKLFPFftScDO}.  \(\T E\) is termed the tangent bundle of \(E\).  Given \(e \in E\), the tangent space at \(e\), denoted \(\T_eE\), is the preimage of \(e\) under the map \(\T E \to E\).  It may be identified with the quotient space \(\lqt{\bb{R}^n}{\Ga_e}\), where \(\Ga_e\) is the orbifold group at \(e\).  A section of the tangent bundle is then simply a continuous map \(X: E \to \T E\) such that for each local chart \(\Xi = \lt(U, \Ga, \tld{U}, \ch\rt)\) for \(E\), there is a smooth, \(\Ga\)-invariant, vector field \(\tld{X}: \tld{U} \to \T \tld{U}\) such that the following diagram commutes:
\ew
\bcd[column sep = 13mm, row sep = 3mm]
\tld{U} \ar[rr, ^^22 \mla{\tld{X}} ^^22] \ar[dd, ^^22 \mla{\tld{\ch}} ^^22] & & \T \tld{U} \ar[dd]\\
& &\\
U \ar[rr, ^^22 \mla{X} ^^22] &  & \pi^{-1}(U)
\ecd
\eew
In a similar way, one may define the cotangent bundle of an orbifold, tensor bundles, bundles of exterior forms etc., with sections of these bundles defined analogously.  \(E\) is orientable if \(\ww{n}\T^* E\) is trivial, i.e.\ has a non-vanishing section (in particular, all the orbifold groups of \(E\) lie in \(\GL_+(n;\bb{R})\)); an orientation is simply a choice of trivialisation.  A \g\ 3-form on an oriented orbifold \(E\) is a section of \(\ww{3}\T^* E\) such that each local representation \(\tld{\ph}\) of \(\ph\) is a \g\ 3-form in the classical (manifold) sense.  In particular, the orbifold group \(\Ga_e\) at each point lies in the stabiliser of \(\tld{\ph}|_0\) and is thus isomorphic to a subgroup of \g.  Riemannian metrics and other geometric structures can be defined similarly.

\subsection{The moduli space \(\cal{G}^{TF}_2\lt(\M_\Ga\rt)\)}

Let:
\ew
\ph_0 = \dd x^{123} + \dd x^{145} + \dd x^{167} + \dd x^{246} - \dd x^{257} - \dd x^{347} - \dd x^{356}
\eew
denote the standard, flat \g\ 3-form on \(\bb{R}^7\) (viewed as a manifold) and consider the orbifold \(\M_\Ga = \lqt{\bb{T}^7}{\Ga}\) for \(\Ga \pc \SL(7;\bb{Z}) \sdp \bb{T}^7\) a finite subgroup of automorphisms of \(\bb{T}^7\).  If \(\ph\) is a torsion-free \g\ 3-form on \(\M_\Ga\), then \(\ph\) lifts to define a \(\Ga\)-invariant torsion-free \g\ 3-form \(\tld{\ph}\) on \(\bb{T}^7\), which is necessarily constant (\wrt\ the usual parallelism of \(\bb{T}^7\)) by a standard `Bochner-type' argument, since \(b^1(\bb{T}^7) = 7 = \dim(\bb{T}^7)\); see \cite[Thm.\ 3.5.4 and 3.5.5]{CMwSH} for details.  Thus \(\tld{\ph} = F^*\ph_0\) for some \(F \in \GL_+(7;\bb{R})\).  Conversely, given \(F \in \GL_+(7;\bb{R})\), the \g\ 3-form \(F^*\ph_0\) descends to \(\M_\Ga\) \iff\ \(\Ga\) preserves \(F^*\ph_0\).  This is equivalent to the condition that for all \(\mc{A} = (A,t) \in \Ga \pc \SL(7;\bb{Z}) \sdp \bb{T}^7\), \(A^*F^*\ph_0 = F^*\ph_0\), i.e.\ that \(FAF^{-1} \in \Gg_2\).  Thus, writing \(\fr{p}_1:\SL(7;\bb{Z}) \sdp \bb{T}^7 \to \SL(7;\bb{Z})\) for the projection homomorphism and defining:
\ew
G^{\Gg_2}_\Ga = \lt\{F \in \GL_+(n;\bb{R}) ~\m|~ F\fr{p}_1(\Ga)F^{-1} \pc \Gg_2 \rt\}
\eew
it has been established that the map:
\ew
\bcd[row sep = 0pt]
\be:G^{\Gg_2}_\Ga \ar[r] &\cal{G}^{TF}_2(\M_\Ga)\\
F \ar[r, maps to] & F^*\ph_0
\ecd
\eew
is surjective.  Call \(\M_\Ga\) a Joyce orbifold if \(G^{\Gg_2}_\Ga \ne \es\), equivalently if \(\M_\Ga\) admits torsion-free \g-structures.

Next, note that \g\ acts on \(G^{\Gg_2}_\Ga\) on the left, and that the map \(\be\) is invariant under this action.  Moreover, the automorphism group of \(\M_\Ga\) is \(\Norm_{\SL(7;\bb{Z}) \sdp \bb{T}^7}(\Ga)\) (i.e.\ the normaliser of \(\Ga\) in \(\SL(7;\bb{Z}) \sdp \bb{T}^7\)) where \(\mc{A} \in \Norm_{\SL(7;\bb{Z}) \sdp \bb{T}^7}(\Ga) \cc \SL(7;\bb{Z}) \sdp \bb{T}^7\) acts via the diagram:
\ew
\bcd
\bb{T}^7 \ar[r, ^^22\mc{A}^^22] \ar[d, ^^22 quot ^^22] & \bb{T}^7 \ar[d, ^^22 quot ^^22]\\
\lqt{\bb{T}^7}{\Ga} \ar[r] & \lqt{\bb{T}^7}{\Ga}
\ecd
\eew
Then \(\fr{p}_1\lt(\Norm_{\SL(7;\bb{Z}) \sdp \bb{T}^7}(\Ga)\rt)\) acts on \(G^{\Gg_2}_\Ga\) on the right, and the map \(\be\) is invariant under this action.  It follows that the moduli space of torsion-free \g-structures on \(\M_\Ga\) is given by:
\ew
\cal{G}^{TF}_2(\M_\Ga) \cong \bqt{G^{\Gg_2}_\Ga}{\Gg_2}{\fr{p}_1\lt(\Norm_{\SL(7;\bb{Z}) \sdp \bb{T}^7}(\Ga)\rt)}.
\eew
(Cf.\ \cite[p.\ 314]{TESaMfFT} for a similar discussion of flat metrics on tori.)

\subsection{Type-decomposition, Hodge decomposition and Hitchin functionals}

Consider the action of \(\Gg_2 = \Stab_{\GL_+(7;\bb{R})}(\vph_0)\) on \(\bb{R}^7\).  The spaces \(\ww{0}\lt(\bb{R}^7\rt)^*\), \(\ww{1}\lt(\bb{R}^7\rt)^*\), \(\ww{6}\lt(\bb{R}^7\rt)^*\) and \(\ww{7}\lt(\bb{R}^7\rt)^*\) are simple \g-modules, however the remaining exterior powers of \(\lt(\bb{R}^7\rt)^*\) are reducible, decomposing into simple \g-modules as:
\e\label{G2TD}
\ww{2}\lt(\bb{R}^7\rt)^* &= \ww[7]{2}\lt(\bb{R}^7\rt)^* \ds \ww[14]{2}\lt(\bb{R}^7\rt)^*;\\
\ww{3}\lt(\bb{R}^7\rt)^* &= \ww[1]{3}\lt(\bb{R}^7\rt)^* \ds \ww[7]{3}\lt(\bb{R}^7\rt)^* \ds \ww[27]{3}\lt(\bb{R}^7\rt)^*;\\
\ww{4}\lt(\bb{R}^7\rt)^* &= \ww[1]{4}\lt(\bb{R}^7\rt)^* \ds \ww[7]{4}\lt(\bb{R}^7\rt)^* \ds \ww[27]{4}\lt(\bb{R}^7\rt)^*;\\
\ww{5}\lt(\bb{R}^7\rt)^* &= \ww[7]{5}\lt(\bb{R}^7\rt)^* \ds \ww[14]{5}\lt(\bb{R}^7\rt)^*,
\ee
where the subscript in each case denotes the dimension of the module and:
\e\label{G2TDEx}
\begin{aligned}
\ww[7]{2}\lt(\bb{R}^7\rt)^* &= \lt\{v \hk \vph_0 ~\m|~ v \in \bb{R}^7 \rt\} \cong \bb{R}^7;\\
\ww[14]{2}\lt(\bb{R}^7\rt)^* &= \lt\{\al \in \ww{2}\lt(\bb{R}^7\rt)^*~\m|~\al \w \Hs_{\vph_0} \vph_0 = 0\rt\} \cong \fr{g}_2;\\
\ww[1]{3}\lt(\bb{R}^7\rt)^* &= \bb{R} \1 \vph_0 \cong \bb{R};\\
\ww[7]{3}\lt(\bb{R}^7\rt)^* &= \lt\{v \hk \Hs_{\vph_0} \vph_0 ~\m|~ v\in \bb{R}^7 \rt\} \cong \bb{R}^7;\\
\ww[27]{3}\lt(\bb{R}^7\rt)^* &= \lt\{a \in \ww{3}\lt(\bb{R}^7\rt)^* ~\m|~ a\w \vph_0 = 0 \,\text{and}\, a \w \Hs_{\vph_0}\vph_0 = 0\rt\} \cong \ss[0]{2}\lt(\bb{R}^7\rt)^*.
\end{aligned}
\ee
(The notation \(\ss[0]{2}\lt(\bb{R}^7\rt)^*\) refers to the space of trace-free symmetric bilinear forms on \(\bb{R}^7\), where the trace is computed using the Euclidean metric.)  Any two simple modules of a given dimension are isomorphic; in particular:
\ew
\Hs_{\vph_0}:\ww[q]{p}\lt(\bb{R}^7\rt)^* \oto{~\mns{\cong}~} \ww[q]{7-p}\lt(\bb{R}^7\rt)^*
\eew
is a \g-equivariant isomorphism for each \(p = 0,...,7\) and corresponding \(q\); thus explicit expressions for all the simple \g-modules occurring in \eref{G2TD} can be deduced from \eref{G2TDEx} via \(\Hs_{\vph_0}\).  As is customary, I write \(\pi_d:\ww{\pt}\lt(\bb{R}^7\rt)^*\to\ww[d]{\pt}\lt(\bb{R}^7\rt)^*\) for the \(g_{\vph_0}\)-orthogonal projection; since for each exterior power no subscript occurs more than once, no ambiguity should arise from this notation.  Exterior forms lying in a particular space \(\ww[q]{p}\lt(\bb{R}^7\rt)^*\) are said to have a definite type and for an arbitrary \(p\)-form \(\si\), the decomposition \(\si = \sum_q \pi_q(\si)\) is called the type-decomposition of \(\si\).

Now let \(\M\) be an oriented 7-orbifold and let \(\ph\) be a \g\ 3-form on \(\M\).  Given a chart \(\Xi = \lt(U, \Ga, \tld{U}, \ch\rt)\) for \(\M\), writing \(\tld{\ph}\) for the lift of \(\ph\) to \(\tld{U}\) as in \sref{Orb}, \(\tld{\ph}\) induces a \(\Ga\)-invariant decomposition \(\ww{p}\T^*\tld{U} \cong \Ds_q \ww[q]{p}\T^*\tld{U}\) and hence one naturally obtains a decomposition \(\ww{p}\T^*\M = \Ds_q \ww[q]{p}\T^*\M\), where each \(\ww[q]{p}\T^*\M\) is an orbifold vector bundle over \(\M\), with fibre over \(x\) isomorphic to \(\lqt{\ww[q]{p}\lt(\bb{R}^7\rt)^*}{\Ga_x}\) (which is well-defined, since \(\Ga_x\) is isomorphic to a subgroup of \g).  Thus \(\ph\) naturally induces a type-decomposition on the exterior forms of \(\M\).

If \(\ph\) is torsion-free and \(\M\) is closed, then this type-decomposition is closely related to the Hodge decomposition of exterior forms on \(\M\).  Recall that the usual Hodge decomposition:
\ew
\Om^p(\M) = \cal{H}^p(\M) \ds \De\Om^p(\M) = \cal{H}^p(\M) \ds \dd\Om^{p-1}(\M) \ds \dd^*\Om^{p+1}(\M)
\eew
is valid on closed orbifolds, where \(\cal{H}^p(\M)\) denotes the space of harmonic \(p\)-forms, \(\De\) denotes the Hodge Laplacian and \(\dd^*\) denotes the co-exterior derivative defined by the metric \(g_\ph\) (see \cite[\S7]{TDTfVM}).  The following result, proved for manifolds in \cite[Thm.\ 3.5.3]{CMwSH}, is easily seen to also hold on orbifolds:

\begin{Thm}\label{RefHD}
For any \(q\):
\ew
\De\circ\pi_q = \pi_q\circ\De.
\eew
In particular, the Hodge decomposition on \(\M\) can be refined to give:
\ew
\Om^p(\M) = \Ds_{q} \lt(\cal{H}^p_q(\M) \ds \De\Om^p_q(\M)\rt),
\eew
where \(\cal{H}^p_q(\M) = \cal{H}^p(\M)\cap\Om^p_q(\M)\) is the space of harmonic \(k\)-forms of type \(q\).
\end{Thm}

The relationship between Hodge Theory and type-decomposition can be made yet more explicit since, as in K\"{a}hler geometry, one can decompose the exterior derivative according to type. Indeed, writing \(\ps = \Hs_\ph\ph\), define the following `refined' exterior differential operators:
\ew
\dd^1_7:\Om^0(\M)&\to\Om^1(\M) & \dd^7_7:\Om^1(\M)&\to\Om^1(\M) & \dd^7_{14}:\Om^1(\M)&\to\Om^2_{14}(\M)\\
f &\mt \dd f                                         & \al &\mt \Hs_\ph\dd(\al\w\ps)               & \al &\mt \pi_{14}(\dd\al)\\
\\
\dd^7_{27}:\Om^1(\M)&\to\Om^3_{27}(\M) & \dd^{14}_{27}:\Om^2_{14}(\M)&\to\Om^3_{27}(\M) & \dd^{27}_{27}:\Om^3_{27}(\M)&\to\Om^3_{27}(\M)\\
\al &\mt \pi_{27}\dd\Hs_\ph(\al\w\ps)                & \be &\mt \pi_{27}(\dd\be)                                             & \ga &\mt \Hs_\ph\pi_{27}(\dd\be).\\
\eew
Note that \(\dd^7_7\) and \(\dd^{27}_{27}\) are both formally \(L^2\)-self-adjoint. Analogously, define \(\dd^7_1=(\dd^1_7)^*\), \(\dd^{14}_7 = (\dd^7_{14})^*\), \(\dd^{27}_7=(\dd^7_{27})^*\) and \(\dd^{27}_{14}=(\dd^{14}_{27})^*\).  The following result, stated in \cite[\S5]{RoG2S} for manifolds, is easily seen to also hold in the orbifold setting described above:
\begin{Thm}\label{G2KI}
Let \(\M\) be a closed, oriented 7-orbifold and \(\ph\) a torsion-free \g\ 3-form on \(\M\).  Then all exterior and co-exterior derivatives on \(\M\) can be expressed purely in terms of the operators \(d^1_7\), \(\dd^7_1\), \(\dd^7_7\), \(\dd^7_{14}\), \(\dd^{14}_7\), \(\dd^7_{27}\), \(\dd^{27}_7\), \(\dd^{14}_{27}\), \(\dd^{27}_{14}\) and \(\dd^{27}_{27}\). In particular, \(\De\) can be expressed in terms of the same operators.
\end{Thm}

\noindent The explicit formulae are presented in \aref{G2-Kahler-Id}.  They will be needed for calculations in \srefs{THoCHaCP&tmu3I} \& \ref{THoCCHaCP&tmu4I}.

Now recall the functionals \(\CH\) and \(\CCH\) described in the introduction.  Firstly, observe that the subset \([\ph]_+ \pc [\ph]\) is open in the \(C^0\)-topology, a fact which follows from the compactness of \(\M\) and a property of \g\ 3-forms termed stability by Hitchin \cite{SF&SM}.  Explicitly, it can be shown that the \(\GL_+(7;\bb{R})\)-orbit of \(\vph_0\) in \(\ww{3}\lt(\bb{R}^7\rt)^*\) is open and thus that all sufficiently small perturbations of a \g\ 3-form are also of \g-type.  As a consequence, one can identify \(\T_{\ph'}[\ph]_+ \cong \dd\Om^2(\M)\).  Using type decomposition, one can explicitly compute the functional derivatives of \(\CH\); this was accomplished by Hitchin in \cite[Thm.\ 19 \& Lem.\ 20]{TGo3Fi6&7D}.\footnote{Note that the formulae for \(\mc{D}\CH|_\ph\) and \(\mc{D}^2\CH|_\ph\) differ from those in \cite{TGo3Fi6&7D}, as the author of this thesis has discovered an error in the numerical factor of \(\frac{7}{18}\) used {\it op.\ cit.}, which has been corrected to \(\frac{1}{3}\) in the formulae here presented.}  Analogous arguments can be used to compute the first and second derivatives of \(\CCH\); see \cite[Thm.\ 1]{SF&SM} and also \cite[Prop.\ 3.3 \& 3.4]{STBoaMLCoG2S}.
\begin{Prop}\label{Hitchin-derivs}
The first and second derivatives of \(\CH\) and \(\CCH\) are given by.
\ew
\bcd[row sep = -2mm, column sep = 5mm]
\mc{D}\CH|_\ph:\dd\Om^2(\M) \ar[r] & \bb{R} & & \mc{D}^2\CH|_\ph:\dd\Om^2(\M) \x \dd\Om^2(\M) \ar[r] & \bb{R}\\
\si \ar[r, maps to] & \frac{1}{3}\bigint_\M \si \w \Hs_\ph \ph & & (\si_1,\si_2) \ar[r, maps to] & \frac{1}{3}\bigint_\M \si_1 \w \Hs_\ph I_\ph(\si_2)
\ecd
\eew
and:
\ew
\bcd[row sep = -2mm, column sep = 5mm]
\mc{D}\CCH|_\ps:\dd\Om^3(\M) \ar[r] & \bb{R} & & \mc{D}^2\CCH|_\ps:\dd\Om^3(\M) \x \dd\Om^3(\M) \ar[r] & \bb{R}\\
\vpi \ar[r, maps to] & \frac{1}{4}\bigint_\M \vpi \w \Hs_\ps\ps & & (\vpi_1\vpi_2) \ar[r, maps to] & \frac{1}{4}\bigint_\M \vpi_1 \w \Hs_\ps J_\ps(\vpi_2)
\ecd
\eew
where:
\e\label{IJ}
I_\ph(\si) = \frac{4}{3}\pi_1(\si) + \pi_7(\si) - \pi_{27}(\si) \et J_\ps(\si) = \frac{3}{4}\pi_1(\si) + \pi_7(\si) - \pi_{27}(\si).
\ee
(Here, the projections \(\pi_\pt\) are defined \wrt\ \(\ph\) and \(\ps\) respectively.)  In particular, for both functionals, the critical points correspond to torsion-free \g\ 3- and 4-forms, as appropriate.\\
\end{Prop}

\section{A spectral generalisation of Morse indices to infinite dimensions}\label{SMI}

Recall the following classical definition \cite[\S2]{MT}:

\begin{Defn}
Let \(b \in \ss{2}\bb{A}^*\) be a symmetric bilinear form on a finite-dimensional real vector space \(\bb{A}\).  The index of \(b\) is the dimension of any maximal subspace \(\bb{B} \pc \bb{A}\) such that \(b|_\bb{B}\) is negative definite.  Equivalently, using a choice of inner product on \(\bb{A}\), one may regard \(b\) as a self-adjoint linear map \(b^{\mla{\sharp}}:\bb{A} \to \bb{A}\); then the index of \(b\) is simply the number of negative eigenvalues of \(b^{\mla{\sharp}}\).

Now let \(\N\) be a finite-dimensional manifold, let \(f:\N\to\bb{R}\) be a Morse function (i.e.\ a function with only non-degenerate critical points) and let \(p\in\N\) be a critical point of \(f\).  The Morse index of \(f\) at \(p\) is the index of the symmetric bilinear form \(D^2f|_p \in \ss{2}\T^*_p\N\).
\end{Defn}

In this section, I use the results of \cite{TSTfVM, OeI} (see also \cite{CPoaEO, APSI, APSII, APSIII}) to propose an extension of this definition to infinite dimensions, resulting in the notion of spectral Morse indices.

Let \((\N,h)\) be a closed, oriented, Riemannian orbifold of odd dimension \(n\) equipped with a real orbifold vector bundle \(E\) with metric \(h^E\), and let \(A\) be a smooth, elliptic, real, self-adjoint pseudodifferential operator of positive order \(m\) acting on sections of \(E\) (see \cite[Defn.\ 1.2]{OeI} for the definition of pseudodifferential operators on orbifolds).  Then \(A\) defines a densely-defined, closed, self-adjoint linear operator on \(L^2(\N,E)\), where the \(L^2\)-norm is defined using the metrics \(h\), \(h^E\).  Define the spectral \(\ze\)- and \(\eta\)- functions of \(A\) to be the partial functions:
\ew
\bcd[row sep = 0pt]
\ze_A: \bb{C} \ar[r, harpoon] & \bb{C} & \eta_A : \bb{C} \ar[r, harpoon] & \bb{C}\\
s \ar[r, maps to] &  \sum_{\la\in\Spec(A)\osr\{0\}} |\la|^{-s} & s \ar[r, maps to] & \sum_{\la\in\Spec(A)\osr\{0\}}\sgn{\la}|\la|^{-s},
\ecd
\eew
defined wherever the sums converge absolutely and locally uniformly.  It follows from \cite{TSTfVM,OeI} that:
\begin{Thm}\label{S-APS}
For \(\N\), \(h\), \(E\), \(h^E\) and \(A\) as above, the spectral \(\ze\)- and \(\eta\)-functions \(\ze_A\) and \(\eta_A\) converge absolutely and locally uniformly on the region:
\ew
\lt\{ s \in \bb{C} ~\m|~ \cal{R}e(s)>\frac{n}{m}\rt\}
\eew
and admit meromorphic continuations to all of \(\bb{C}\) which are holomorphic on a neighbourhood 0; let \(\ze(A)\) and \(\eta(A)\) denote their respective values at \(0\).  Then \(\ze(A), \eta(A) \in \bb{R}\), and for any \(\ell > 0\):
\ew
\ze(\ell A) = \ze(A) \et \eta(\ell A) = \eta(A).
\eew
Moreover, the maps:
\ew
\bcd[row sep = 0pt]
\ze: \Ps^{>0}_{\text{inv-sa}}(\N;E) \ar[r] & \bb{R} & \eta:\Ps^{>0}_{\text{inv-sa}}(\N;E) \ar[r] &\bb{R}\\
A \ar[r, maps to] & \ze_A(0) & A \ar[r, maps to]& \eta_A(0)
\ecd
\eew
are smooth, where \(\Ps^{>0}_{\text{inv-sa}}(\N;E)\) denotes the space of (smooth) invertible, real, self-adjoint pseudodifferential operators of positive order acting on \(E\).
\end{Thm}

Using \tref{S-APS}, I make the following definition:
\begin{Defn}\label{Spec-Morse}
Let \(\N\), \(h\), \(E\), \(h^E\) and \(A\) be as above.  I define the spectral Morse index of \(A\) to be:
\ew
\SI(A) = \frac{\ze(A) - \eta(A)}{2}.
\eew
Then \(\SI(A)\) is real and invariant under rescalings \(A \mt \ell A\) for \(\ell > 0\).  Moreover, \(\SI\) defines a smooth map:
\ew
\SI:\Ps^{>0}_{\text{inv-sa}}(\N;E)&\to\bb{R}.
\eew
\end{Defn}

The motivation for \dref{Spec-Morse} can be understood as follows: define the spectral Morse function of \(A\) to be:
\ew
\bcd[row sep = 0pt]
\mu_A: \lt\{s \in \bb{C} ~\m|~ \fr{Re}s > \frac{n}{m} \rt\} \ar[r] &\bb{C}\\
s \ar[r, maps to] & \sum_{\substack{\la\in\Spec(A)\\\la<0}}|\la|^{-s}.
\ecd
\eew
Then, by \tref{S-APS}, \(\mu_A\) admits an analytic continuation to all of \(\bb{C}\) and \(\mu_A(0) = \SI(A)\).  If \(A\) has only finitely many negative eigenvalues, then the sum defining \(\mu_A\) converges on all of \(\bb{C}\) and \(\mu_A(0)\) is simply the number of negative eigenvalues of \(A\).  Thus in general, one should think of \(\SI(A)\) as a regularised measure of the `number of negative eigenvalues of \(A\)'.\\

\section{\(\mu_3\): Morse indices of the critical points of \(\CH\)}\label{THoCHaCP&tmu3I}

The aim of this section is to prove that the critical points of the functional \(\CH\) have well-defined spectral Morse indices.  Let \(\M\) be a closed, oriented 7-orbifold and let \(\ph\) be a torsion-free \g\ 3-form on \(\M\).  Since \(\CH\) is diffeomorphism invariant, it induces a functional \(\CH': \rqt{[\ph]_+}{\Diff_0(\M)} \to (0,\infty)\).  The following result generalises \cite[Thm.\ 19 and Prop.\ 21]{TGo3Fi6&7D} to the case of orbifolds, as well as rephrasing the argument {\it op.\ cit.}\ to obtain an explicit expression for \(\mc{D}\CH'\) transverse to the action of diffeomorphisms:
\begin{Prop}\label{CH-Prop}
The tangent space \(\T_\ph\lt(\rqt{[\ph]_+}{\Diff_0(\M)}\rt)\) can formally be identified with the space:
\ew
\dd^*\Om^3(\M)\cap\Om^2_{14}(\M).
\eew
Moreover, using the natural \(L^2\) inner product on \(\dd^*\Om^3(\M)\cap\Om^2_{14}(\M)\) induced by \(\ph\), the Hessian \(\mc{D}^2\CH'|_\ph\) can formally be identified with the invertible, linear map:
\ew
\frac{1}{3}\mc{E}(\ph) = -\frac{1}{3}\dd^*\dd:\dd^*\Om^3(\M)\cap\Om^2_{14}(\M) \to \dd^*\Om^3(\M)\cap\Om^2_{14}(\M),
\eew
where \(\mc{E}(\ph)\) denotes the restriction of the operator \(-\dd^*\dd\) to the domain \(\dd^*\Om^3(\M)\cap\Om^2_{14}(\M)\).  In particular, the critical points of \(\CH'\) are non-degenerate.
\end{Prop}

\begin{proof}
Recall that \(\T_\ph[\ph]_+\) is simply \(\dd\Om^2(\M)\), by the stability of \g\ 3-forms.  Let \(X\in\Ga(\M,\T\M)\) be a vector-field on M.  The Lie derivative of \(\ph\) along \(X\) can be computed using Cartan's formula \cite[Prop.\ 2.25(d)]{FoDM&LG} to be:
\ew
\cal{L}_X\ph = X\hk\dd\phi + \dd(X\hk\ph) = \dd(X\hk\ph),
\eew
since \(\dd\ph=0\).  Thus, as \(X\) varies, the space of Lie derivatives \(\cal{L}_X\ph\), and hence the tangent space to the orbit of \(\Diff_0(\M)\) through \(\ph\), is precisely the space \(\dd\Om^2_7(\M)\) (cf.\ \eref{G2TDEx}).

Next, I describe the tangent space \(\T_\ph\lt(\rqt{[\ph]_+}{\Diff_0(\M)}\rt)\).  Recall that the usual Hodge decomposition:
\ew
\Om^p(\M) = \cal{H}^p(\M) \ds \dd\Om^{p-1}(\M) \ds \dd^*\Om^{p+1}(\M)
\eew
is valid on closed orbifolds.  Using the isomorphism:
\ew
\bcd
\dd^*\Om^3(\M) \ar[r, shift left = 1.5mm, "\dd"] & \dd\Om^2(\M) \ar[l, shift left = 1.5mm, "\dd^*G"]
\ecd
\eew
(where \(G\) is the Green's operator for the Hodge Laplacian \(\De\) induced by \(\ph\)) I identify \(\T_\ph[\ph]_+\cong\dd^*\Om^3(\M)\) and:
\ew
\T_\ph\Diff_0(\M) \cong \dd^*G\dd\Om^2_7(\M) = \dd^*\dd\Om^2_7(\M)\subset\dd^*\Om^3(\M),
\eew
where the middle equality uses that \(G\) commutes with type-decomposition, since \(\ph\) is torsion-free (see \tref{RefHD}).  Thus, one can identify \(\T_\ph\lt(\rqt{[\ph]_+}{\Diff_0(\M)}\rt)\) with the \(L^2\)-orthocomplement of \(\dd^*\dd\Om^2_7(\M)\) in \(\dd^*\Om^3(\M)\).  Explicitly, writing \(\<,\?\) for the \(L^2\) inner product on forms induced by \(g_\ph\), given \(\ga \in \dd^*\Om^3(\M)\) and \(\de \in \Om^2_7(\M)\), one computes that:
\ew
\<\ga, \dd^*\dd\de\? = \<\dd^*\dd\ga, \de\? = \<\De\ga,\de\? = \<\ga,\De\de\?
\eew
and thus \(\ga \in \dd^*\dd\Om^2_7(\M)^\bot\) \iff\ \(\ga\hs{0.7mm}\bot\hs{0.7mm}\De\Om^2_7(\M)\).  Using the refined Hodge decomposition (see \tref{RefHD}):
\ew
\Om^2_7(\M) = \cal{H}^2_7(\M) \ds \De\Om^2_7(\M),
\eew
together with the fact that \(\ga \in \dd^*\Om^3(\M)\) is automatically orthogonal to \(\cal{H}^2_7(\M)\), it follows that \(\ga \in \dd^*\dd\Om^2_7(\M)^\bot\) \iff\ \(\ga\hs{0.7mm}\bot\hs{0.7mm}\Om^2_7(\M)\) and thus:
\ew
\dd^*\dd\Om^2_7(\M)^\bot = \dd^*\Om^3(\M) \cap \Om^2_{14}(\M).
\eew

Using this description, together with \pref{Hitchin-derivs} and Stokes' Theorem, the second functional derivative of \(\CH'\) at \(\ph\) is:
\ew
\bcd[row sep = 0pt]
\mc{D}^2\CH'|_\ph: \lt(\dd^*\Om^3(\M)\cap\Om^2_{14}(\M)\rt)^2 \ar[r] & \bb{R}\\
(\ga_1, \ga_2) \ar[r, maps to] & \frac{1}{3} \bigint_\M \ga_1 \w\Hs_\ph(\dd^*I\dd\ga_2).
\ecd
\eew
Using \eref{14KI}, one can compute that for \(\ga \in \dd^*\Om^3(\M)\cap\Om^2_{14}(\M)\):
\ew
\dd^*I\dd\ga = -\dd^*\dd\ga.
\eew
Thus, writing \(\<,\?\) for the \(L^2\) inner product on \(\dd^*\Om^3(\M)\) induced by \(\ph\), it follows that:
\ew
\mathcal{D}^2\CH'|_\ph(\ga_1,\ga_2) = -\frac{1}{3}\<\dd\ga_1,\dd\ga_2\?,
\eew
as claimed.

\end{proof}

In light of \pref{CH-Prop}, and motivated by Morse theory, it is natural to ask whether the critical point \(\ph\) has a well-defined notion of Morse index.  Clearly the classical Morse index of \(\ph\) is infinite, since \(\mc{D}^2\CH'|_\ph\) is negative definite.  Nevertheless, it is possible to `regularise' the Morse index of \(\ph\), since the operator \(\frac{1}{3}\mc{E}(\ph)\) has a well-defined spectral Morse index (equivalently, one may consider \(\mc{E}(\ph)\), since spectral Morse indices are scale invariant).  Consider the second-order pseudodifferential operator acting on \(\Om^2(\M)\) via:
\ew
\cal{E}(\ph) = \pi_{harm,\ph} + \De + 2\dd^* I\dd,
\eew
where \(\pi_{harm,\ph}\) denotes the \(L^2\)-orthogonal projection onto \(\ph\)-harmonic forms.  With respect to the decomposition:
\ew
\Om^2(\M) = \cal{H}^2(\M) \ds \dd\Om^1(\M) \ds \dd^*\dd\Om^2_7(\M) \ds \lt[\dd^*\Om^3(\M)\cap\Om^2_{14}(\M)\rt]
\eew
obtained in the proof of \pref{CH-Prop}, the operator \(\cal{E}(\ph)\) acts diagonally, given explicitly by:
\ew
\cal{E}(\ph) =
\begin{cases}
\Id & \text{ on } \cal{H}^2(\M);\\
\dd\dd^* & \text{ on } \dd\Om^1(\M);\\
\dd^*\dd & \text{ on } \dd^*\dd\Om^2_7(\M);\\
-\dd^*\dd & \text{ on } \dd^*\Om^3(\M)\cap\Om^2_{14}(\M).
\end{cases}
\eew
In particular, the operator \(\cal{E}(\ph)\) is invertible and self-adjoint, and has the same negative spectrum as the operator \(\mc{E}(\ph)\) defined in \pref{CH-Prop}.  Thus, by the discussion after \dref{Spec-Morse}, the sum:
\ew
\bcd[row sep = 0pt]
\mu_\ph:\lt\{s \in \bb{C} ~\m|~ \fr{Re}s > \frac{7}{2} \rt\} \ar[r] &\bb{C}\\
s \ar[r, maps to] & \sum_{\substack{\la\in\Spec(\mc{E}(\ph))\\\la<0}}|\la|^{-s}
\ecd
\eew
converges absolutely and locally uniformly, and admits a meromorphic continuation to all of \(\bb{C}\) which is holomorphic at \(0\).  Moreover, the value at \(0\) is simply \(\SI(\cal{E}(\ph))\), and since \(\cal{E}(\ph)\) depends smoothly on \(\ph\) and \(\SI: \Ps^{>0}_{inv-sa} \to \bb{R}\) is smooth, it follows that \(\mu_\ph(0)\) depends smoothly on \(\ph\).  Thus, I obtain:
\begin{Thm}\label{mu3-def-thm}
Let \(\M\) be a closed, oriented 7-orbifold and let \(\cal{G}_2^{TF}(\M)\) denote the moduli space of torsion-free \g-structures on \(\M\).  Define the \(\mu_3\)-invariant of a torsion-free \g\ 3-form \(\ph\) to be the value of the meromorphic function \(\mu_\ph\) at \(0\).  Then \(\mu_3\) is diffeomorphism invariant, invariant under rescaling \(\ph \mt \ell^3 \ph\) for \(\ell > 0\) and defines a smooth map:
\ew
\mu_3: \cal{G}_2^{TF}(\M) \to \bb{R}.
\eew
\end{Thm}

\begin{proof}
The only statement which remains to be proven is that \(\mu_3\lt(\ell^3\ph\rt) = \mu_3(\ph)\).  As \(g_{\ell^3\ph} = \ell^2g_\ph\), it follows from \cite[p.\ 306]{OtHE&tIT} that \(\dd^*_{\ell^3\ph} = \ell^{-2}\dd^*_\ph\).  Thus the negative spectrum of \(\cal{E}\lt(\ell^3\ph\rt)\) coincides with the negative spectrum of \(\ell^{-2}\cal{E}(\ph)\) (even though \(\cal{E}\lt(\ell^3\ph\rt) \ne \ell^{-2}\cal{E}(\ph)\), due to the presence of \(\pi_{harm, \ph}\) in the definition of \(\cal{E}(\ph)\)).  The result now follows from the scale invariance of \(\SI\).

\end{proof}
~

\section{\(\mu_4\): Morse indices of the critical points of \(\CCH\)}\label{THoCCHaCP&tmu4I}

The aim of this section is to prove that the critical points of the functional \(\CCH\) also have well-defined spectral Morse indices.  Let \(\M\) be a closed, oriented 7-orbifold, let \(\ps\) be a torsion-free \g\ 4-form on \(\M\) and write \(\CCH'\) for the functional \(\rqt{[\ps]_+}{\Diff_0(\M)} \to (0,\infty)\) induced by \(\CCH\).  The Hessian \(\mc{D}^2\CCH'|_\ps\) is completely described via the following result:

\begin{Prop}\label{CCH-Prop}
Write \(\Om^3_{1 \ds 27}(\M)\) as a shorthand for \(\Om^3_1(\M) \ds \Om^3_{27}(\M)\).  Then the tangent space \(\T_\ps\lt(\rqt{[\ps]_+}{\Diff_0(\M)}\rt)\) can formally be identified with the space:
\ew
\dd^*\Om^4(\M) \cap \Om^3_{1 \ds 27}(\M).
\eew
Moreover, there is an \(L^2\)-orthogonal decomposition:
\ew
\dd^*\Om^4(\M) \cap \Om^3_{1 \ds 27}(\M) &= \lt\{\om\in\dd^*\Om^4(\M)\cap\Om^3_{1 \ds 27}(\M)~\middle|~\pi_{27}\dd\om = 0\rt\} \ds \lt(\dd^*\Om^4(\M) \cap \Om^3_{27}(\M)\rt)\\
&= \mc{S}^+_4(\ps) \ds \mc{S}^-_4(\ps)
\eew
and, using the \(L^2\) inner product, \(\mc{D}^2\CCH'|_\ps\) can formally be identified with the invertible linear map \(\frac{1}{4}\dd^*\dd \ds -\frac{1}{4}\dd^*\dd\) on \(\mc{S}^+_4(\ps) \ds \mc{S}^-_4(\ps)\); denote this map by \(\frac{1}{4}\mc{F}(\ps)\).  In particular, \(\mc{D}^2\CCH'|_\ps\) is positive/negative definite on \(\mc{S}^\pm_4(\ps)\), respectively, and the critical points of \(\CCH'\) are non-degenerate saddles.
\end{Prop}

\begin{proof}
As in the proof of \pref{CH-Prop}, one can identify \(\T_\ps[\ps]_+\) and \(\T_\ps \lt(\Diff_0(\M)\1\ps\rt)\) with the spaces \(\dd^*\Om^4(\M)\) and \(\dd^*\dd\Om^3_7(\M)\), respectively. Hence one can identify \(\T_\ps\lt(\rqt{[\ps]_+}{\Diff_0(\M)}\rt)\) with the \(L^2\)-orthocomplement of \(\dd^*\dd\Om^3_7(\M)\) in \(\dd^*\Om^4(\M)\), {\it viz.}:
\ew
\dd^*\Om^4(\M) \cap \Om^3_{1 \ds 27}(\M).
\eew

Using this description, together with \pref{Hitchin-derivs} and Stokes' Theorem, the second functional derivative of \(\CCH'\) at \(\ps\) is:
\ew
\bcd[row sep = 0pt]
\mc{D}^2\CCH'|_\ps: \lt(\dd^*\Om^4(\M) \cap \Om^3_{1 \ds 27}(\M)\rt)^2 \ar[r] & \bb{R}\\
(\om_1,\om_2) \ar[r, maps to] & \frac{1}{4}\bigint_\M \om_1 \w\Hs_\ps(\dd^*J\dd\om_2),
\ecd
\eew
where \(J = \frac{3}{4}\pi_1 + \pi_7 - \pi_{27}\) was defined in \eref{IJ}.  To further analyse \(\mc{D}^2\CCH'|_\ps\), I prove:

\begin{Cl}
There is an \(L^2\)-orthogonal decomposition:
\ew
\dd^*\Om^4(\M) \cap \Om^3_{1 \ds 27}(\M) = \underbrace{\lt\{\om\in\dd^*\Om^4(\M)\cap\Om^3_{1 \ds 27}(\M)~\middle|~\pi_{27}\dd\om = 0\rt\}}_{= \mc{S}^+_4(\ps)} \ds \underbrace{\lt\{\om\in\dd^*\Om^4(\M)\cap\Om^3_{1 \ds 27}(\M)~\middle|~\pi_7\dd\om = 0\rt\}}_{= \mc{S}^-_4(\ps)}.
\eew
Moreover:
\e\label{equiv-of-expressions}
\lt\{\om\in\dd^*\Om^4(\M)\cap\Om^3_{1 \ds 27}(\M)~\middle|~\pi_7\dd\om = 0\rt\} = \dd^*\Om^4(\M)\cap \Om^3_{27}(\M).
\ee
\end{Cl}

\begin{proof}[Proof of Claim]
Recall that in the statement of \tref{G2KI}, there are no operators of the form \(d^1_1\) and \(d^{27}_1\). This implies, in particular, that:
\ew
\dd\Big(\Om^3_{1 \ds 27}(\M)\Big) \subset \Om^4_7(\M) \ds \Om^4_{27}(\M)
\eew
and hence the spaces \(\dd\mc{S}^+_4(\ps)\pc\Om^4_7(\M)\) and \(\dd\mc{S}^-_4(\ps)\pc\Om^4_{27}(\M)\) are \(L^2\)-orthogonal. Using \tref{RefHD}, one can also verify that \(\dd^*\dd\mc{S}^\pm_4(\ps) = \De\mc{S}^\pm_4(\ps) = \mc{S}^\pm_4(\ps)\). Thus:
\ew
\mc{S}^+_4(\ps) \text{ and } \mc{S}^-_4(\ps) \text{ are } L^2\text{-orthogonal} & \LE \dd^*\dd\mc{S}^+_4(\ps) \text{ and } \mc{S}^-_4(\ps) \text{ are } L^2\text{-orthogonal}\\
&\LE \dd\mc{S}^+_4(\ps) \text{ and } \dd\mc{S}^-_4(\ps) \text{ are } L^2\text{-orthogonal},
\eew
so \(\mc{S}^+_4(\ps)\) and \(\mc{S}^-_4(\ps)\) are indeed \(L^2\)-orthogonal, as claimed. To complete the proof, therefore, it suffices to prove that each \(\om\in\dd^*\Om^4(\M)\cap\Om^3_{1 \ds 27}(\M)\) can be written as \(\om = \om^++\om^-\), for some \(\om^\pm\in\mc{S}^\pm_4(\ps)\), and to verify \eref{equiv-of-expressions}.

Given \(\om\in \dd^*\Om^4(\M) \cap \Om^3_{1 \ds 27}(\M)\), write \(\om = f\ph + \ga\) for some unique \(f\in\Om^0(\M)\) and \(\ga\in\Om^3_{27}(\M)\).  Note that one can trivially write:
\e\label{S-split}
\om =\bigg (f\ph + \frac{7}{12}\dd^7_{27}\dd^{27}_7G\ga\bigg) + \bigg(\ga - \frac{7}{12}\dd^7_{27}\dd^{27}_7G\ga\bigg) = \om^+ + \om^-,
\ee
where \(G\) denotes the Green's operator for the Hodge Laplacian defined by \(\ps\). I claim that \(\om^\pm\in\mc{S}^\pm_4(\ps)\), i.e.:
\ew
\om^\pm \in \dd^*\Om^4(\M) \cap \Om^3_{1 \ds 27}(\M), \hs{5mm} \pi_{27}\dd\om^+ = 0 \et \pi_7\dd\om^- = 0.
\eew

Begin with the first of these points.  Since clearly \(\om^\pm \in \Om^3_{1 \ds 27}(\M)\), it \stp\ \(\om^\pm \in \dd^*\Om^4(\M)\).  Since \(\om = f\ph + \ga \in \dd^*\Om^4(\M) \pc \De\Om^3(\M)\), it follows that \(f\ph \in \De\Om^3_1(\M)\) by \tref{RefHD} and hence it is orthogonal to \(\cal{H}^3_1(\M)\); likewise \(\ga \in \De\Om^3_{27}(\M)\) is orthogonal to \(\cal{H}^3_{27}(\M)\).  Moreover:
\ew
\frac{7}{12}\dd^7_{27}\dd^{27}_7G\ga = \frac{7}{12}G \dd^7_{27}\dd^{27}_7\ga \in \De\Om^3_{27}(\M)
\eew
and hence \(\frac{7}{12}\dd^7_{27}\dd^{27}_7G\ga\) is also orthogonal to \(\cal{H}^3_{27}(\M)\).  It follows that \(\om^\pm\) are each orthogonal to \(\cal{H}^3(\M)\) and so to prove that \(\om^\pm \in \dd^*\Om^4(\M)\), it \stp\ that \(\dd^*\om^\pm = 0\).

In general, given \(f' \in \Om^0(\M)\) and \(\ga' \in \Om^3_{27}(\M)\), by \erefs{1KI} and \eqref{27KI} the condition \(\dd^*(f'\ph + \ga') = 0\) is equivalent to the pair of equations:
\e\label{d*=0}
\dd^{27}_7\ga' = 3\dd f' \et \dd^{27}_{14}\ga' = 0.
\ee
Since \(\om = f\ph + \ga \in \dd^*\Om^4(\M)\), it follows that \(\dd^{27}_7\ga = 3\dd f\) and \(\dd^{27}_{14}\ga = 0\).  Therefore:
\begin{alignat*}{2}
\dd^{27}_7 \lt(\frac{7}{12}\dd^7_{27}\dd^{27}_7G\ga\rt) &= \dd^{27}_7\lt(\De G\ga - (\dd^{27}_{27})^2G\ga\rt) & \hs{5mm} & \text{(using \eref{27-De} and \(\dd^{27}_{14}\ga = 0\))}\\
&= \dd^{27}_7\underbrace{\De G\ga}_{=\ga} - \underbrace{(\dd^{27}_7\dd^{27}_{27})}_{=\frac{1}{2}\dd^7_7\dd^{27}_7}\dd^{27}_{27}G\ga & \hs{5mm} & \text{(using \(\ga\hs{0.7mm}\bot\hs{0.7mm}\cal{H}^3_{27}(\M)\) and \eref{d2=0})}\\
& = \dd^{27}_7\ga - \frac{1}{2}\dd^7_7\underbrace{(\dd^{27}_7\dd^{27}_{27})}_{= -\frac{3}{2}\dd^{14}_7\dd^{27}_{14}}G\ga & \hs{5mm} & \text{(using two subequations from \eref{d2=0})}\\
&= \dd^{27}_7\ga & \hs{5mm} & \text{(using \(\dd^{27}_{14}G\ga = G\dd^{27}_{14}\ga = 0\))}.
\end{alignat*}

Likewise:
\begin{alignat*}{2}
\dd^{27}_{14} \lt(\frac{7}{12}\dd^7_{27}\dd^{27}_7G\ga\rt) &= \underbrace{\dd^{27}_{14}\ga}_{=0} - \underbrace{\lt(\dd^{27}_{14}\dd^{27}_{27}\rt)}_{=-\frac{1}{4}\dd^7_{14}\dd^{27}_7}\dd^{27}_{27}G\ga & \hs{5mm} & \text{(using \eref{d2=0})}\\
&= \frac{1}{4}\dd^7_{14}\underbrace{\lt(\dd^{27}_7 \dd^{27}_{27}\rt)}_{= -\frac{3}{2}\dd^{14}_7\dd^{27}_{14}} G\ga & \hs{5mm} & \text{(using two subequations from \eref{d2=0})}\\
&= 0 & \hs{5mm} & \text{(using \(\dd^{27}_{14}G\ga = G\dd^{27}_{14}\ga = 0\))}.
\end{alignat*}

It follows from these last two calculations, together with the conditions \(\dd^{27}_7\ga = 3\dd f\) and \(\dd^{27}_{14}\ga = 0\), that \(\om^\pm\) each satisfy \eref{d*=0} and hence \(\dd^*\om^\pm = 0\).  Thus, all that remains is to prove \(\pi_{27}\dd\om^+ = 0\) and \(\pi_7\dd\om^- = 0\).

Using \eref{27KI}, one computes:
\ew
\pi_{27}\dd\om^+ &= \frac{7}{12}\Hs_\ph\underbrace{\lt(\dd^{27}_{27}\dd^7_{27}\rt)}_{\mfn{=\frac{1}{2}\dd^7_{27}\dd^7_7}}\dd^{27}_7G\ga \hs{5mm} \text{(by \eref{d2=0})}\\
&= \frac{7}{24}\Hs_\ph\dd^7_{27}\underbrace{\lt(\dd^7_7\dd^{27}_7\rt)}_{\mfn{= -3\dd^{14}_7\dd^{27}_{14}}}G\ga \hs{5mm} \text{(by \eref{d2=0})}\\
&= -\frac{7}{8}\Hs_\ph\dd^7_{27}\dd^{14}_7\dd^{27}_{14}G\ga = 0,
\eew
since \(\dd^{27}_{14}G\ga = G\dd^{27}_{14}\ga = 0\).  Similarly:
\ew
\pi_7\dd\om^- = \frac{1}{4}\dd^{27}_7\lt[\ga - \frac{7}{12}\dd^7_{27}\dd^{27}_7G\ga\rt] \w \ph = 0,
\eew
since \(\dd^{27}_7\lt(\frac{7}{12}\dd^7_{27}\dd^{27}_7G\ga\rt) = \dd^{27}_7\ga\), as above.  Thus \(\om^\pm \in \mc{S}^\pm_4(\ps)\), as claimed.

Finally, to verify that \(\mc{S}^-_4(\ps) = \dd^*\Om^4(\M) \cap \Om^3_{27}(\M)\), note that \(\mc{S}^-_4(\ps) \cc \dd^*\Om^4(\M) \cap \Om^3_{27}(\M)\) follows by \eref{S-split}.  Conversely, if \(\ga \in \dd^*\Om^4(\M) \cap \Om^3_{27}(\M)\), then \(\dd^*\ga = 0\) forces \(\dd^{27}_7\ga = 0\) by \eref{d*=0} and hence \(\pi_7\dd\ga = 0\) (see \eref{27KI}).  Thus \(\ga \in \mc{S}^-_4(\ps)\), completing the proof of the claim.

\let\qed\relax
\end{proof}

Given this claim, \tref{CCH-Prop} follows swiftly.  Indeed, recalling the definition of \(J\) in \eref{IJ}, it follows that for \(\om\in\dd^*\Om^4(\M)\cap\Om^3_{1 \ds 27}(\M)\):
\ew
\dd^*J\dd\om = 
\begin{cases}
+\dd^*\dd\om \in \mc{S}^+_4(\ps) & \text{ if } \om \in \mc{S}^+_4(\ps);\\
-\dd^*\dd\om \in \mc{S}^-_4(\ps) & \text{ if } \om \in \mc{S}^-_4(\ps).\\
\end{cases}
\eew
Thus, the symmetric bilinear form \(D^2\CCH'|_\ps\) is given by:
\ew
D^2\CCH'|_\ps(\om_1,\om_2) =
\begin{cases}
+\frac{1}{4}\<\dd\om_1,\dd\om_2\? & \text{ if } \om_1, \om_2 \in \mc{S}^+_4(\ps);\\
-\frac{1}{4}\<\dd\om_1,\dd\om_2\? & \text{ if } \om_1, \om_2 \in \mc{S}^-_4(\ps),
\end{cases}
\eew
where \(\<,\?\) denotes the \(L^2\) inner product on \(\dd^*\Om^4(\M)\) induced by \(\ps\).

\end{proof}

It follows from previous work of the author (see \cite[Thm.\ 3.9]{UA&BotDHFoG2&SG2F}) that both \(\mc{S}^\pm\) are infinite dimensional.  In particular, the classical Morse index of \(\ps\) is, as for \(\CH\), infinite.  Nevertheless, it is, again, possible to define the regularised Morse index of \(\ps\).  Consider the second-order pseudodifferential operator:
\ew
\cal{F}(\ps) = \pi_{harm,\ps} + \De + 2\dd^*J\dd
\eew
where \(\pi_{harm,\ps}\) denotes the \(L^2\)-orthogonal projection onto \(\ps\)-harmonic forms.  With respect to the decomposition:
\ew
\Om^3 = \cal{H}^3(\M) \ds \dd\Om^2(\M) \ds \dd^*\dd\Om^3_7(\M) \ds \mc{S}^+_4(\ps) \ds \mc{S}^-_4(\ps)
\eew
obtained in the proof of \pref{CCH-Prop}, the operator \(\cal{F}(\ps)\) acts diagonally, given explicitly by:
\ew
\cal{F}(\ph) =
\begin{cases}
\Id & \text{ on } \cal{H}^3(\M);\\
\dd\dd^* & \text{ on } \dd\Om^2(\M);\\
\dd^*\dd & \text{ on } \dd^*\dd\Om^3_7(\M);\\
3\dd^*\dd & \text{ on } \mc{S}^+_4(\ps);\\
-\dd^*\dd & \text{ on } \mc{S}^-_4(\ps).
\end{cases}
\eew
In particular, the operator \(\cal{F}(\ps)\) is invertible and self-adjoint, and has the same negative spectrum as the operator \(\mc{F}(\ps)\) defined in \pref{CCH-Prop}.  Thus, by the discussion after \dref{Spec-Morse}, the sum:
\ew
\bcd[row sep = 0pt]
\mu_\ps:\lt\{s \in \bb{C} ~\m|~ \fr{Re}s > \frac{7}{2} \rt\} \ar[r] &\bb{C}\\
s \ar[r, maps to] & \sum_{\substack{\la\in\Spec(\mc{F}(\ps))\\\la<0}}|\la|^{-s}
\ecd
\eew
converges absolutely and locally uniformly, and admits a meromorphic continuation to all of \(\bb{C}\) which is holomorphic at \(0\).  Moreover, the value at \(0\) is simply \(\SI(\cal{F}(\ps))\) and since \(\cal{F}(\ps)\) depends smoothly on \(\ps\), and \(\SI: \Ps^{>0}_{inv-sa} \to \bb{R}\) is smooth, it follows that \(\mu_\ps(0)\) depends smoothly on \(\ps\).  Moreover \(\mu_\ps(0)\) is scale invariant by the same argument as for \(\mu_3\).  Thus, it has been shown that:
\begin{Thm}\label{mu4-def-thm}
Let \(\M\) be a closed, oriented 7-orbifold and let \(\cal{G}_2^{TF}(\M)\) denote the moduli space of torsion-free \g-structures on \(\M\).  Define the \(\mu_4\)-invariant of a torsion-free \g\ 4-form \(\ps\) to be the value of the meromorphic function \(\mu_\ps\) at \(0\).  Then \(\mu_4\) is diffeomorphism invariant, invariant under rescaling \(\ps \mt \ell^4 \ps\) for \(\ell > 0\) and defines a smooth map:
\ew
\mu_4: \cal{G}_2^{TF}(\M) \to \bb{R}.
\eew\\
\end{Thm}

\section{Computing the eigenvalues of \(\mc{E}(\ph)\) and \(\mc{F}(\ps)\) on Joyce orbifolds}

This is the first of two sections which aim to prove \tref{Joyce-comp-thm}.  Let \(\M_\Ga\) be a Joyce orbifold, let \(\ph\) be a (constant) torsion-free \g\ 3-form on \(\M_\Ga\) and let \(\ps\) denote the corresponding \g\ 4-form.  Recall from \sref{THoCHaCP&tmu3I} that \(\mu_3(\ph)\) is the value at 0 of the meromorphic extension of:
\ew
\bcd[row sep = 0pt]
\mu_\ph:\lt\{s \in \bb{C} ~\m|~ \fr{Re}s > \frac{7}{2} \rt\} \ar[r] &\bb{C}\\
s \ar[r, maps to] & \sum_{\substack{\la\in\Spec(\mc{E}(\ph))\\\la<0}}|\la|^{-s}
\ecd
\eew
where \(\mc{E}(\ph)\) acts on \(\dd^*\Om^3(\M_\Ga) \cap \Om^2_{14}(\M_\Ga)\) via \(-\dd^*\dd\).  Thus the task is to explicitly compute the spectrum of \(\mc{E}(\ph)\).  Since exterior forms on \(\M_\Ga\) are equivalent to \(\Ga\)-invariant exterior forms on \(\bb{T}^7\) and since \(-\dd^*\dd\) is a real operator, by elliptic regularity this is equivalent to computing the spectrum of \(-\dd^*\dd\) acting on the complex Hilbert space:
\ew
\bb{H}^\Ga = \lt(\dd^*H^1\Om^3(\bb{T}^7)_\bb{C} \cap L^2\Om^2_{14}(\bb{T}^7)_\bb{C}\rt)^\Ga,
\eew
where \((-)_\bb{C} = (-) \ts_\bb{R} \bb{C}\), \(L^2\) and \(H^1\) denote Lebesgue and Sobolev spaces of sections, respectively, and \(\lt(-\rt)^\Ga\) denotes the \(\Ga\)-invariant subspace.

To this end, identify \(\lt(\T_0\bb{T}^7\rt)_\bb{C} \cong \lt(\bb{R}^7\rt)_\bb{C}\) and recall that every \(\om \in \ww{\pt}\lt(\bb{R}^7\rt)^*_\bb{C}\) defines a left-invariant, complex exterior form on \(\bb{T}^7\), which I also denote by \(\om\).  This defines a natural embedding \(\ww{\pt}\lt(\bb{R}^7\rt)^*_\bb{C} \emb \Om^\pt(\bb{T}^7)_\bb{C}\) onto the space of constant (equivalently, \(\ph\)-harmonic) complex exterior forms on \(\bb{T}^7\).  Let \(\lt(\bb{Z}^7\rt)^*\) denote the dual lattice of \(\bb{Z}^7\) \wrt\ the metric \(g = g_\ph\) induced by \(\ph\), i.e.:
\ew
\lt(\bb{Z}^7\rt)^* = \lt\{ x \in \bb{R}^7 ~\m|~ \text{for all } y \in \bb{Z}^7, ~ g(x,y) \in \bb{Z} \rt\}
\eew
and, given \(l\in\lt(\bb{Z}^7\rt)^*\), define a smooth, \(\bb{C}\)-valued function \(\ch_l:\bb{T}^7\to\bb{C}\) by:
\ew
\ch_l:\bb{T}^7&\to\bb{C}\\
x + \bb{Z}^7 &\mt e^{2\pi i g(l,x)},
\eew
Define:
\ew
\bb{H}_l = \lt\{ \ch_l\cdot \al ~\m|~ \al \in \ww[14]{2}\lt(\bb{R}^7\rt)^*_\bb{C}  \text{ satisfies }l \hk \al = 0\rt\}
\eew
and finally define:
\ew
\mc{L} = \lt\{ -4\pi^2\|l\|^2_g ~\m|~ l \in \lt(\bb{Z}^7\rt)^* \rt\}.
\eew

\begin{Prop}
For each \(\la \in \mc{L}\osr\{0\}\), define:
\ew
\bb{H}(\la) = \Ds_{l \in \lt(\bb{Z}^7\rt)^* : -4\pi^2\|l\|^2_g = \la} \bb{H}_l.
\eew
Then, there is a decomposition:
\ew
\bb{H} = \ol{\Ds_{\la \in \mc{L}\osr\{0\}} \bb{H}(\la)}
\eew
of \(\bb{H}\) into eigenspaces of \(\mc{E}(\ph)\), where \(\mc{E}(\ph)\) acts on \(\bb{H}(\la)\) via \(\la\Id\).
\end{Prop}

\begin{proof}
By the Peter--Weyl theorem \cite{DVdPDeGKG}:
\ew
L^2\Om^2_{14}(\bb{T}^7)_\bb{C} = \ol{\Ds_{l \in \lt(\bb{Z}^7\rt)^*} \lt\{\ch_l \al ~\m|~ \al \in \ww[14]{2}\lt(\bb{R}^7\rt)^*_\bb{C}\rt\}}.
\eew
Given \(\al = \sum_{l \in \lt(\bb{Z}^7\rt)^*} \ch_l \al^l \in L^2\Om^2_{14}(\bb{T}^7)_\bb{C}\), observe that:
\ew
\pi_{harm}\al = \al^0 \in \ww[14]{2}\lt(\bb{R}^7\rt)^*_\bb{C}.
\eew
Similarly, using the identity:
\ew
\dd\ch_l = 2\pi i \ch_l l^\flat ,
\eew
where \({}^\flat:\bb{R}^7 \to \lt(\bb{R}^7\rt)^*\) denotes the musical isomorphism induced by \(g\), one computes that for \(\al \in \Om^2_{14}(\bb{T}^7)_\bb{C}\) smooth:
\ew
\dd^*\al = -2\pi i\sum_{l\in\lt(\bb{Z}^7\rt)^*}\ch_l \big(l\hk\al^l\big).
\eew
Since the space \(\bb{H} = \dd^*H^1\Om^3(\bb{T}^7)_\bb{C} \cap L^2\Om^2_{14}(\bb{T}^7)_\bb{C}\) is simply the closure of the space:
\ew
\lt\{ \al \in \Om^2_{14}(\bb{T}^7)_\bb{C} ~\m|~ \pi_{harm}\al = 0 \text{ and } \dd^*\al = 0 \rt\}
\eew
in the \(L^2\)-norm, it follows that:
\ew
\bb{H} &= \ol{\Ds_{l \in \lt.\lt(\bb{Z}^7\rt)^*\m\osr\{0\}\rt.} \bb{H}_l}\\
&= \ol{\Ds_{\la \in \mc{L}\osr\{0\}} \bb{H}(\la)}.
\eew
Thus to complete the proof, it suffices to note that for \(\ch_l \al^l \in \bb{H}_l\):
\ew
-\dd^*\dd(\ch_l\al^l) = -4\pi^2||l||^2_g\ch_l\al^l,
\eew
which follows from \(l \hk \al = 0\) (see \cite[p.\ 363]{SoGH-LOoFT&RS}).  Thus \(\bb{H}(\la)\) is in fact the \(\la\)-eigenspace of \(-\dd^*\dd\), as required.

\end{proof}

Since \(\Ga\) commutes with the action of \(-\dd^*\dd\), it follows that:
\ew
\bb{H}^\Ga = \ol{\Ds_{\la \in \mc{L}\osr\{0\}} \bb{H}(\la)^\Ga}.
\eew
Thus, for \(\fr{Re}(s) > \frac{7}{2}\), one finds that:
\ew
\mu_{\mc{E}(\ph)}(s) = \sum_{\la \in \mc{L}\osr\{0\}} \frac{\dim \bb{H}(\la)^\Ga}{|\la|^s}.
\eew

The calculation for \(\mu_4\) is closely analogous: firstly note that the negative spectrum of \(\mc{F}(\ps)\) is the same as the spectrum of \(-\dd^*\dd\) acting on the space:
\ew
\lt(\bb{H}'\rt)^\Ga = \lt(\dd^*H^1\Om^4(\bb{T}^7)_\bb{C} \cap L^2\Om^3_{27}(\bb{T}^7)_\bb{C}\rt)^\Ga.
\eew
For \(l \in \lt(\bb{Z}^7\rt)^*\) and \(\la \in \mc{L}\), define:
\ew
\bb{H}'_l = \lt\{ \ch_l\cdot \al ~\m|~ \al \in \ww[27]{3}\lt(\bb{R}^7\rt)^*_\bb{C}  \text{ satisfies }l \hk \al = 0\rt\}
\eew
and:
\ew
\bb{H}'(\la) = \Ds_{l \in \lt(\bb{Z}^7\rt)^* : -4\pi^2\|l\|^2_g = \la} \bb{H}'_l.
\eew
Then, as for \(\mu_3\):
\ew
\lt(\bb{H}'\rt)^\Ga = \ol{\Ds_{\la \in \mc{L}\osr\{0\}} \bb{H}'(\la)^\Ga},
\eew
where \(-\dd^*\dd\) acts on each \(\bb{H}'(\la)\) by \(\la\Id\).  It follows that for \(\fr{Re}s > \frac{7}{2}\):
\ew
\mu_{\mc{F}(\ps)}(s) = \sum_{\la \in \mc{L}\osr\{0\}} \frac{\dim \bb{H}'(\la)^\Ga}{|\la|^s}.
\eew
Thus, the computation of \(\mu_3(\ph)\) and \(\mu_4(\ps)\) has been reduced to the representation-theoretic problem of computing \(\dim \bb{H}(\la)^\Ga\) and \(\dim \bb{H}'(\la)^\Ga\) for each \(\la \in \mc{L}\osr\{0\}\).  This will occupy the next section.\\

\section{Multiplicities of the eigenvalues of \(\mc{E}(\ph)\) and \(\mc{F}(\ps)\)}\label{mu-comp-RepT}

Write \(\rh^{(\la)}\) for the (finite-dimensional) representation of \(\Ga\) on \(\bb{H}(\la) = \Ds_{l \in \lt(\bb{Z}^7\rt)^*: -4\pi^2\|l\|_g^2 = \la} \bb{H}_l\).  Recall that the character \(\ch^{(\la)}:\Ga \to \bb{R}\) of \(\rh^{(\la)}\) is defined by:
\ew
\ch^{(\la)}(\mc{A}) = \Tr_{\bb{H}(\la)}(\rh^{(\la)}(\mc{A})), \hs{2mm} \mc{A} \in \Ga.
\eew
The dimension of \(\bb{H}(\la)^\Ga\) can be computed using \(\ch^{(\la)}\) via the formula \cite[eqn.\ (2.9)]{RTAFC}:
\e\label{inv-dim}
\dim \bb{H}(\la)^\Ga = \frac{1}{|\Ga|}\sum_{\mc{A} \in \Ga} \ch^{(\la)}(\mc{A}).
\ee
Thus, the task is to compute the character \(\ch^{(\la)}\).  This is accomplished by the following proposition:
\begin{Prop}\label{mu3-char}
Given \(A \in \End(\bb{R}^7)\), define:
\ew
\Tr_8^{\SU(3)}(A) = \frac{\Tr_{\bb{R}^7}(A)^2 - \Tr_{\bb{R}^7}(A^2)}{2} - 2\Tr_{\bb{R}^7}(A) + 1.
\eew
Moreover, given \(\la \in \mc{L}\osr\{0\}\) and \(\mc{A} = (A,t) \in \SL(7;\bb{Z}) \sdp \bb{T}^7\), define:
\ew
\mc{G}(\la,\mc{A}) = \lt\{ l\in\lt(\bb{Z}^7\rt)^* ~\m|~ -4\pi^2\|l\|^2_g = \la, Al = l \rt\}.
\eew
Then:
\ew
\ch^{(\la)}(\mc{A}) = \sum_{l \in \mc{G}(\la,\mc{A})} e^{2 \pi i g(l, t)} \Tr^{\SU(3)}_8(A).
\eew
In particular, by \eref{inv-dim}:
\ew
\dim \bb{H}(\la)^\Ga = \frac{1}{|\Ga|} \sum_{\mc{A} = (A,t) \in \Ga} \sum_{l \in \mc{G}(\la,\mc{A})} e^{2 \pi i g(l, t)} \Tr^{\SU(3)}_8(A).
\eew
\end{Prop}

The proof proceeds by a series of lemmas:

\begin{Lem}\label{step-1-decomp}
For each \(\bb{H}_l \pc \bb{H}(\la)\), define a representation \(\rh_l\) of \(\Ga\) on \(\bb{H}_l\) via:
\ew
\rh_l(\mc{A})[u] = \proj_{\bb{H}_l} \lt\{\rh^{(\la)}(\mc{A})[u]\rt\}, \hs{2mm}  \mc{A} \in \Ga, u \in \bb{H}_l,
\eew
where \(\proj_{\bb{H}_l}\) denotes the projection \(\bb{H}(\la) = \Ds_{l' \in \lt(\bb{Z}^7\rt)^*: -4\pi^2\|l'\|_g^2 = \la} \bb{H}_{l'} \to \bb{H}_l\), and write \(\ch_l\) for the corresponding character.  Then for each \(\mc{A} \in \Ga\):
\ew
\ch^{(\la)}(\mc{A}) = \sum_{l\in\mc{G}(\la,\mc{A})} \ch_l(\mc{A}).
\eew
\end{Lem}

\begin{proof}
For all \(\mc{A} \in \Ga\), the linear map \(\rh^{(\la)}(\mc{A})\) is represented by the block matrix:
\ew
\bpm
\rh_{l_1}(\mc{A}) & * & \hdots\\
* & \rh_{l_2}(\mc{A}) & \\
\vdots & & \ddots
\epm
\eew
where \(\bb{H}(\la) = \bb{H}_{l_1} \ds \bb{H}_{l_2} \ds ...\).  In particular:
\ew
\ch^{(\la)} = \sum_{l \in \lt(\bb{Z}^7\rt)^*: -4\pi^2\|l\|_g^2 = \la} \ch_l.
\eew
Now given \(\mc{A} = (A,t) \in \Ga \pc \SL(7;\bb{Z}) \x \bb{T}^7\), for \(\ch_l \al^l = e^{2\pi i g(l,x)}\al^l \in \bb{H}_l\):
\e\label{3act}
\mc{A}^*(e^{2\pi i g(l,x)} \al^l) = e^{2\pi i g(l,t)}e^{2\pi i g(A^Tl,x)}A^*\al^l \in \bb{H}_{A^Tl},
\ee
where \(A^T\) denotes the adjoint of \(A\) \wrt\ \(g\).  Thus \(\rh_l(\mc{A})\) is non-zero \iff\ \(A^Tl = l\), which is equivalent to \(Al = l\) (since \(A\) preserves \(g\)).  This completes the proof.

\end{proof}

\begin{Lem}\label{mu3-traces}
Given \(l \in \mc{G}(\la,\mc{A})\), define:
\ew
\bb{A}_l = \lt\{ \al \in \ww[14]{2}\lt(\bb{R}^7\rt) ~\m|~ l \hk \al = 0\rt\}.
\eew
Then the trace \(\Tr_{\lt(\bb{A}_l\rt)_\bb{C}}(A)\) of \(A = \fr{p}_1(\mc{A})\) acting on \(\lt(\bb{A}_l\rt)_\bb{C}\) via pullback is:
\ew
\Tr_{(\bb{A}_l)_\bb{C}}(A) = \Tr_8^{\SU(3)}(A) = \frac{\Tr_{\bb{R}^7}(A)^2 - \Tr_{\bb{R}^7}(A^2)}{2} - 2\Tr_{\bb{R}^7}(A) + 1.
\eew
\end{Lem}

\begin{proof}
Firstly, note that, since complexification does not affect the trace of a real operator, it is equivalent to compute the trace of \(A\) acting on \(\bb{A}_l\).  Identify \(\Stab_{\GL_+(7;\bb{R})}(\ph)\) with the group \g\ and recall that \(\Stab_{\Gg_2}(l) \cong \SU(3)\) \cite[Prop.\ 2.7]{A3F&ESLGoTG2}.  Let \(\bb{B} = \<l\?^\bot\) with its natural orientation, let \(\th \in \lt(\bb{R}^7\rt)^*\) be a correctly oriented annihilator of \(\bb{B}\) and, using the splitting \(\bb{R}^7 = \<l\? \ds \bb{B}\), write:
\ew
\ph = \th \w \om + \rh.
\eew
Since \(\SU(3) \pc \GL(3;\bb{C})\), \(\bb{B}\) inherits a natural complex structure \(J\) with respect to which \(\om\) is a positive \((1,1)\)-form on \(\bb{B}\) \cite{AEBVPfG2S} and there is an \(\SU(3)\)-invariant decomposition:
\ew
\ww{2}\bb{B}^* = \bb{R}\cdot\om \ds \lt[\ww[8]{1,1}\bb{B}^*\rt] \ds \ls\ww{2,0}\bb{B}^*\rs,
\eew
where \(\lt[\ww[8]{1,1}\bb{B}^*\rt]\) is the orthocomplement to \(\bb{R}\cdot\om\) in \(\lt[\ww{1,1}\bb{B}^*\rt]\) and \(\ls\ww{2,0}\bb{B}^*\rs = \lt\{u \hk \rh ~\m|~ u \in \bb{B}\rt\} \cong \bb{B}\).  (Here, I use the notation from \cite[p.\ 32]{RG&HG} that, for \(r \ne s\), \(\ls\ww{r,s}\lt(\bb{R}^6\rt)^*\rs = \lt(\ww{r,s}\lt(\bb{R}^6\rt)^* \ds \ww{s,r}\lt(\bb{R}^6\rt)^*\rt) \cap \ww{r+s}\lt(\bb{R}^6\rt)^*\) is the set of real forms of type \((r,s)+(s,r)\) and \(\lt[\ww{r,r}\lt(\bb{R}^6\rt)^*\rt] = \ww{r,r}\lt(\bb{R}^6\rt)^* \cap \ww{2r}\lt(\bb{R}^6\rt)^*\) is the set of real forms of type \((r,r)\).)  Define an isomorphism:
\ew
\ch_6: \bb{B}^* &\to \ls\ww{2,0}\bb{B}^*\rs\\
v\hk\om &\mt v\hk\rh.
\eew
Then by \cite[Lem.\ 1]{AEBVPfG2S}:
\ew
\ww[14]{2}\lt(\bb{R}^7\rt)^* = \lt[\ww[8]{1,1}\bb{B}^*\rt] \ds \lt\{2 \th \w \al + \ch_6(\al) ~\m|~\al \in \bb{B}^*\rt\}.
\eew
In particular:
\ew
\bb{A}_l = \lt\{ \al \in \ww[14]{2}\lt(\bb{R}^7\rt) ~\m|~ l \hk \al = 0\rt\} = \lt[\ww[8]{1,1}\bb{B}^*\rt]
\eew
and thus \(\Tr_{\bb{A}_l}(A)\) is simply the trace of \(A\) acting on the space \(\lt[\ww[8]{1,1}\bb{B}^*\rt]\).

Using \cite[Prop.\ 2.1]{RTAFC}, the trace of \(A\) acting on \(\ww{2}\lt(\bb{R}^7\rt)^*\) is:
\ew
\Tr_{\ww{2}\lt(\bb{R}^7\rt)^*}(A) = \frac{\Tr_{\bb{R}^7}(A)^2 - \Tr_{\bb{R}^7}(A^2)}{2}.
\eew
Hence, using \(\ww{2}\lt(\bb{R}^7\rt)^* = \ww[7]{2}\lt(\bb{R}^7\rt)^* \ds \ww[14]{2}\lt(\bb{R}^7\rt)^* \cong \bb{R}^7 \ds \ww[14]{2}\lt(\bb{R}^7\rt)^*\), one finds that:
\ew
\Tr_{\ww[14]{2}\lt(\bb{R}^7\rt)^*}(A) = \frac{\Tr_{\bb{R}^7}(A)^2 - \Tr_{\bb{R}^7}(A^2)}{2} - \Tr_{\bb{R}^7}(A).
\eew
Next, note that \(\Tr_\bb{B}(A) = \Tr_{\bb{R}^7}(A) - 1\), since \(Al = l\).  Thus \(\ww[14]{2}\lt(\bb{R}^7\rt)^* \cong \bb{B} \ds \lt[\ww[8]{1,1}\bb{B}^*\rt]\), yields:
\ew
\Tr_{\lt[\ww[8]{1,1}\bb{B}^*\rt]}(A) = \Tr_{\ww[14]{2}\lt(\bb{R}^7\rt)^*}(A) - \Tr_{\bb{R}^7}(A) + 1 = \frac{\Tr_{\bb{R}^7}(A)^2 - \Tr_{\bb{R}^7}(A^2)}{2} - 2\Tr_{\bb{R}^7}(A) + 1,
\eew
as required.

\end{proof}

\begin{proof}[Proof of \pref{mu3-char}]
By \eref{3act}, it follows that for \(\mc{A} = (A,t) \in \Ga\), \(l \in \mc{G}(\la,\mc{A})\):
\ew
\ch_l(\mc{A}) = e^{2 \pi i g(l,t)} \Tr_{\lt(\bb{A}_l\rt)_\bb{C}}(A) = e^{2 \pi i g(l,t)} \Tr_8^{\SU(3)}(A)
\eew
The result now follows from \lref{step-1-decomp}.

\end{proof}

Using \pref{mu3-char}, it follows that for all \(\fr{Re}(s) > \frac{7}{2}\):
\e\label{initial-sum-3}
\mu_{\mc{E}(\ph)}(s) = \sum_{\la \in \mc{L}\osr\{0\}} \frac{1}{|\la|^s} \lt(\frac{1}{|\Ga|} \sum_{\mc{A} = (A,t) \in \Ga} \sum_{l \in \mc{G}(\la,\mc{A})} e^{2 \pi i g(l, t)} \Tr^{\SU(3)}_8(A)\rt).
\ee
Write \(\mc{G}(\mc{A}) = \lt\{ l \in \lt(\bb{Z}^7\rt)^* ~\m|~ Al = l\rt\}\), a lattice in the 1-eigenspace of \(A\). Note that \(\mc{G}(\mc{A})\) is non-zero, having rank equal to the \(\bb{Q}\)-dimension of the 1-eigenspace of \(A\) viewed as a linear map over \(\bb{Q}\) (recall that \(A \in \SL(7;\bb{Z})\) has integer entries), which is equal to the \(\bb{R}\)-dimension of the 1-eigenspace of \(A\) viewed as a linear map over \(\bb{R}\), which is itself non-zero since \(A\) is orientation-preserving and preserves the metric \(g\) (and the dimension of \(\bb{R}^7\) is odd).  Thus, by rearranging \eref{initial-sum-3}, one finds:
\ew
\mu_{\mc{E}(\ph)}(s) = \frac{1}{(2\pi)^{2s}|\Ga|} \sum_{\mc{A} = (A,t) \in \Ga} \Tr^{\SU(3)}_8(A) \lt(\sum_{l \in \mc{G}(\mc{A})\osr\{0\}} \frac{e^{2 \pi i g(l,t)}}{\|l\|_g^{2s}}\rt).
\eew
The sum:
\ew
\sum_{l \in \mc{G}(\mc{A})\osr\{0\}} \frac{e^{2 \pi i g(l,t)}}{\|l\|_g^{2s}}
\eew
is an example of an Epstein \(\ze\)-function, and hence the value at \(s = 0\) of its meromorphic extension to \(\bb{C}\) is always \(-1\), independent of \(t\) or the rank of the lattice \cite[p.\ 627]{ZTAZI}.  Thus:
\ew
\mu_3(\M_\Ga,\ph) = \frac{-1}{|\Ga|} \sum_{\mc{A} = (A,t) \in \Ga} \Tr^{\SU(3)}_8(A);
\eew
in particular, this formula is independent of \(\ph\).  Thus it has been established:
\begin{Thm}\label{mu3-J-thm}
Let \(\M_\Ga = \lqt{\bb{T}^7}{\Ga}\) be a Joyce orbifold.  The map \(\mu_3:\cal{G}^{TF}_2(\M_\Ga) \to \bb{R}\) is constant, taking the value:
\ew
\mu_3(\M_\Ga) = \frac{-1}{|\Ga|} \sum_{\mc{A} = (A,t) \in \Ga} \Tr^{\SU(3)}_8(A).
\eew
\end{Thm}

Now consider \(\mu_4\).  All of the above analysis is easily adapted to the case of \(\mu_4\) except \lref{mu3-traces}, which must be replaced with the following result:
\begin{Lem}
Given \(l \in \mc{G}(\la,\mc{A})\), define:
\ew
\bb{A}'_l = \lt\{ \al \in \ww[27]{3}\lt(\bb{R}^7\rt)_\bb{C} ~\m|~ l \hk \al = 0\rt\}.
\eew
Then the trace \(\Tr_{\lt(\bb{A}'_l\rt)_\bb{C}}(A)\) of \(A = \fr{p}_1(\mc{A})\) acting on \(\lt(\bb{A}'_l\rt)_\bb{C}\) via pullback is:
\ew
\Tr_{(\bb{A}'_l)_\bb{C}}(A) = \Tr_{12}^{\SU(3)}(A) = \frac{\Tr_{\bb{R}^7}(A)^3 + 2\Tr_{\bb{R}^7}(A^3) - 3\Tr_{\bb{R}^7}(A^2)\Tr_{\bb{R}^7}(A)}{6} - \frac{\Tr_{\bb{R}^7}(A)^2 - \Tr_{\bb{R}^7}(A^2)}{2} - 2.
\eew
\end{Lem}

\begin{proof}
As before, note that \(\Tr_{(\bb{A}'_l)_\bb{C}}(A) = \Tr_{\bb{A}'_l}(A)\).  Let \(\bb{B}\), \(\th\), \(\rh\), \(\om\) and \(J\) be as in the proof of \lref{mu3-traces}.  Then there is a decomposition:
\ew
\ww{3}\bb{B}^* = \underbrace{\bb{R}\cdot\lt\<\rh,J^*\rh\rt\?}_{\mns{\ls\ww{(3,0)}\bb{B}^*\rs}} \ds \underbrace{\ls\ww[6]{2,1}\bb{B}^*\rs \ds \ls\ww[12]{2,1}\bb{B}^*\rs}_{\mns{\ls\ww{(2,1)}\bb{B}^*\rs}}
\eew
into simple \(\SU(3)\)-modules, where \(\ls\ww[6]{2,1}\bb{B}^*\rs = \lt\{ \vth \w \om ~\m|~ \vth \in \bb{B}^*\rt\} \cong \bb{B}\) and \(\ls\ww[12]{2,1}\bb{B}^*\rs\) denotes the orthocomplement to \(\ls\ww[6]{2,1}\bb{B}^*\rs\) in \(\ls\ww{2,1}\bb{B}^*\rs\).
Define an isomorphism:
\ew
\tld{\ch_6}: \bb{B}^* &\to \ls\ww{2,0}\bb{B}^*\rs\\
u\hk\om &\mt u\hk J_\rh^*\rh.
\eew

Then one can verify that:
\ew
\ww[27]{3}\bb{A}^* &= \bb{R}\cdot\lt(4\th\w\om - 3\rh\rt) \ds \lt\{\th\w\tld{\ch}_6(\al) - \al\w\om~\middle|~\al\in\bb{B}^*\rt\} \ds \th \w \lt[\ww[8]{1,1}\bb{B}^*\rt] \ds \ls\ww[12]{2,1}\bb{B}^*\rs
\eew
and hence:
\ew
\bb{A}'_l = \lt\{ \al \in \ww[14]{2}\lt(\bb{R}^7\rt) ~\m|~ l \hk \al = 0\rt\} = \ls\ww[12]{2,1}\bb{B}^*\rs.
\eew

One can compute directly that:
\ew
\Tr_{\ww{3}\lt(\bb{R}^7\rt)^*}(A) = \frac{\Tr_{\bb{R}^7}(A)^3 + 2\Tr_{\bb{R}^7}(A^3) - 3\Tr_{\bb{R}^7}(A^2)\Tr_{\bb{R}^7}(A)}{6}.
\eew
Hence, using the \g-invariant decomposition \(\ww{3}\lt(\bb{R}^7\rt)^* \cong \bb{R} \ds \bb{R}^7 \ds \ww[27]{3}\lt(\bb{R}^7\rt)^*\), one finds that:
\ew
\Tr_{\ww[27]{3}\lt(\bb{R}^7\rt)^*}(A) = \frac{\Tr_{\bb{R}^7}(A)^3 + 2\Tr_{\bb{R}^7}(A^3) - 3\Tr_{\bb{R}^7}(A^2)\Tr_{\bb{R}^7}(A)}{6} - \Tr_{\bb{R}^7}(A) - 1.
\eew
Now, since there is an \(\SU(3)\)-invariant decomposition \(\ww[27]{3}\lt(\bb{R}^7\rt)^* \cong \bb{R} \ds \bb{B} \ds \lt[\ww[8]{1,1}\bb{B}^*\rt] \ds \ls\ww[12]{2,1}\bb{B}^*\rs\), it follows that:
\ew
\Tr_{\ls\ww[12]{2,1}\bb{B}^*\rs}(A) &= \Tr_{\ww[27]{3}\lt(\bb{R}^7\rt)^*}(A) - 1 - \Tr_\bb{B}(A) - \Tr_{\lt[\ww[8]{1,1}\bb{B}^*\rt]}(A)\\
&= \Tr_{\ww[27]{3}\lt(\bb{R}^7\rt)^*}(A) - \Tr_{\bb{R}^7}(A) - \Tr^{\SU(3)}_8(A).
\eew
The result follows.

\end{proof}

Arguing as for \(\mu_3\), one obtains:
\begin{Thm}\label{mu4-J-thm}
Let \(\M_\Ga = \lqt{\bb{T}^7}{\Ga}\) be a Joyce orbifold.  Then the \(\mu_4\)-invariant \(\mu_4:\cal{G}^{TF}_2(\M_\Ga) \to \bb{R}\) is constant, taking the value:
\ew
\mu_4(\M_\Ga) = \frac{-1}{|\Ga|} \sum_{\mc{A} = (A,t) \in \Ga} \Tr^{\SU(3)}_{12}(A).
\eew\\
\end{Thm}

\section{Numerical values of \(\mu_3\) and \(\mu_4\) on explicit examples of Joyce orbifolds}

Using the closed formulae for \(\mu_3\) and \(\mu_4\) given in \trefs{mu3-J-thm} and \ref{mu4-J-thm}, many explicit examples of \(\mu_3\) and \(\mu_4\) can be computed.  I give a few examples below:

\begin{Ex}[Flat Tori] Firstly, consider the case \(\Ga = {\bf 1}\). Then:
\ew
\mu_3(\bb{T}^7) = -\Tr^{\SU(3)}_8(\Id) = -8
\eew
and:
\ew
\mu_4(\bb{T}^7) = -\Tr^{\SU(3)}_{12}(\Id) = -12.
\eew
(Note that \(\Tr^{\SU(3)}_8(\Id) = \dim \lt[\ww[8]{1,1}\bb{B}^*\rt]\) and \(\Tr^{\SU(3)}_{12}(\Id) = \dim \ls\ww[12]{2,1}\bb{B}^*\rs\), as expected.)
\end{Ex}

For the first non-trivial case, let me consider a family of examples in \cite[\S3.1]{CR7MwHG2II}. Consider \(\tld{\Ga} = \<\al,\be,\ga\? \pc \lt(\Gg_2 \cap \SL(7;\bb{Z})\rt) \sdp \bb{T}^7\), where:
\caw
\al: \lt(x^1, x^2, x^3, x^4, x^5, x^6, x^7\rt) \mt \lt(x^1, x^2, x^3, -x^4, -x^5, -x^6, -x^7\rt)\\
\be: \lt(x^1, x^2, x^3, x^4, x^5, x^6, x^7\rt) \mt \lt(x^1, -x^2, -x^3, x^4, x^5, b^6 - x^6, b^7 - x^7\rt)\\
\ga: \lt(x^1, x^2, x^3, x^4, x^5, x^6, x^7\rt) \mt \lt(-x^1, x^2, c^3 - x^3, x^4, c^5 - x^5, x^6, c^7 - x^7\rt)
\caaw
for suitable \(b^6, b^7, c^3, c^5, c^7 \in \lt\{0,\frac{1}{2}\rt\}\).  (Note that these formulae differ from those in \cite{CR7MwHG2II} since a different `standard \g\ 3-form' is used {\it op.\ cit.}.  Joyce's original notation is obtained upon applying the transformation \(x^i \mt -x^{8-i}\).)   Then it is shown in \cite{CR7MwHG2II} that \(\tld{\Ga}\cong \lt(\rqt{\bb{Z}}{2}\rt)^3\), generated by \(\al\), \(\be\) and \(\ga\). One can compute that for all \(\mc{A} = (A,t)\in\tld{\Ga}\osr\{\Id\}\), \(A\) is diagonal, with diagonal entries (in some order):
\ew
1,1,1,-1,-1,-1,-1.
\eew
Thus, one can verify that for all \(\mc{A} = (A,t) \in \tld{\Ga}\osr\{\Id\}\):
\ew
\Tr^{\SU(3)}_8(A) = 0 \et \Tr^{\SU(3)}_{12}(A) = 4.
\eew
Using this, one can compute further examples:
\begin{Ex}[K3 Orbifold] Take \(\Ga_1 = \<\al\? \pc \tld{\Ga}\). Then \(\M_1 = \M_{\Ga_1} \cong \lqt{\bb{T}^4}{\lt(\rqt{\bb{Z}}{2}\rt)} \x \bb{T}^3\), where \(\lqt{\bb{T}^4}{\lt(\rqt{\bb{Z}}{2}\rt)}\) is the standard orbifold used in the Kummer construction of the K3 surface. Using \trefs{mu3-J-thm} and \ref{mu4-J-thm}, one can compute that:
\ew
\mu_3(\M_1) = \frac{-1}{2}(8 + 0) = -4
\eew
and:
\ew
\mu_4(\M_1) = \frac{-1}{2}(12 + 4) = -8.
\eew
\end{Ex}

\begin{Ex}[Calabi--Yau Orbifold] Set \(\lt(b^6,b^7\rt) = \lt(0,\frac{1}{2}\rt)\) and take \(\Ga_2 = \<\al,\be\?\pc\tld{\Ga}\). Then \(\M_2 = \M_{\Ga_2} \cong \lqt{\bb{T}^6}{\lt(\rqt{\bb{Z}}{2}\rt)^2} \x S^1\), where the \(\SU(3)\)-orbifold \(\lqt{\bb{T}^6}{\lt(\rqt{\bb{Z}}{2}\rt)^2}\) admits a smooth Calabi--Yau 3-fold as a crepant resolution. Then:
\ew
\mu_3(\M_2) = \frac{-1}{4}(8 + 3\x0) = -2
\eew
and:
\ew
\mu_4(\M_2) = \frac{-1}{4}(12 + 3\x 4) = -6.
\eew
\end{Ex}

\begin{Ex}[\g\ Orbifold] Now consider the full group \(\Ga_3 = \tld{\Ga}\). For suitable choices of \(b^i\) and \(c^j\), the orbifold \(\M_3 = \M_{\Ga_3}\) may be resolved to form a smooth \g-manifold (see \cite{CR7MwHG2II}). Then:
\ew
\mu_3(\M_3) = \frac{-1}{8}(8 + 7\x0) = -1
\eew
and:
\ew
\mu_4(\M_3) = \frac{-1}{8}(12 + 7\x 4) = -5.
\eew
\end{Ex}
Using similar methods, many further explicit examples can now easily be computed.

\begin{Rk}
In \cite{AAIoG2M}, \CGN\ define a different spectral invariant of torsion-free \g-structures on manifolds, denoted \(\ol{\nu}\).  By \cite[Thm.\ 7.7]{OeI}, \(\ol{\nu}\) is equally well-defined on closed \g-orbifolds.  Moreover, for any closed \g-orbifold \((\M,\ph)\) which admits an orientation-reversing isometry, \(\ol{\nu}(\ph) = 0\) (cf.\ \cite[Prop.\ 1.5(iii)]{AAIoG2M}).

Now consider the torsion-free \g-structure \(\ph_0\) on the orbifolds \(\bb{T}^7\), \(\M_1\) and \(\M_2\) above.  Each of \((\bb{T}^7,\ph_0)\), \((\M_1,\ph_0)\) and \((\M_2,\ph_0)\) admit orientation-reversing isometries, since each orbifold is the Riemannian product of \(S^1\) with a 6-orbifold.  Thus:
\ew
\ol{\nu}(\bb{T}^7,\ph_0) = \ol{\nu}(\M_1,\ph_0) = \ol{\nu}(\M_2,\ph_0) = 0;
\eew
in particular, the \(\ol{\nu}\)-invariant alone cannot distinguish between these three \g-orbifolds.  By contrast:
\caw
\mu_3(\bb{T}^7) = -8, \hs{5mm} \mu_3(\M_1) = -4 \et \mu_3(\M_2) = -2\\
\mu_4(\bb{T}^7) = -12, \hs{5mm} \mu_4(\M_1) = -8 \et \mu_4(\M_2) = -6
\caaw
and thus either \(\mu_3\) or \(\mu_4\) alone is sufficient to distinguish between the orbifolds \(\bb{T}^7\), \(\M_1\) and \(\M_2\).  This suggests that \(\mu_3\) and \(\mu_4\) may well be better tools for distinguishing between Joyce manifolds than the currently existing \(\ol{\nu}\)-invariant.\\
\end{Rk}

\appendix

\section{Formulae for refined exterior derivatives induced by torsion-free \g\ 3-forms}\label{G2-Kahler-Id}

Recall that on an oriented 7-orbifold \(\M\) with torsion-free \g\ 3-form \(\ph\), the usual exterior derivative can be decomposed according to type, yielding the `refined' exterior differential operators:
\ew
\dd^1_7:\Om^0(\M)&\to\Om^1(\M) & \dd^7_7:\Om^1(\M)&\to\Om^1(\M) & \dd^7_{14}:\Om^1(\M)&\to\Om^2_{14}(\M)\\
f &\mt \dd f                                         & \al &\mt \Hs_\ph\dd(\al\w\ps)               & \al &\mt \pi_{14}(\dd\al)\\
\\
\dd^7_{27}:\Om^1(\M)&\to\Om^3_{27}(\M) & \dd^{14}_{27}:\Om^2_{14}(\M)&\to\Om^3_{27}(\M) & \dd^{27}_{27}:\Om^3_{27}(\M)&\to\Om^3_{27}(\M)\\
\al &\mt \pi_{27}\dd\Hs_\ph(\al\w\ps)                & \be &\mt \pi_{27}(\dd\be)                                             & \ga &\mt \Hs_\ph\pi_{27}(\dd\be).\\
\eew
Analogously, define \(\dd^7_1=(\dd^1_7)^*\), \(\dd^{14}_7 = (\dd^7_{14})^*\), \(\dd^{27}_7=(\dd^7_{27})^*\) and \(\dd^{27}_{14}=(\dd^{14}_{27})^*\), where \({}^*\) denotes the formal \(L^2\) adjoint (\(\dd^7_7\) and \(\dd^{27}_{27}\) are both formally \(L^2\) self-adjoint).  Then, the main result of \cite[\S5]{RoG2S} can readily be generalised to orbifolds, giving:
\begin{Thm*}
All exterior and co-exterior derivatives on \((\M,\ph)\) can be expressed purely in terms of the operators \(d^1_7\), \(\dd^7_1\), \(\dd^7_7\), \(\dd^7_{14}\), \(\dd^{14}_7\), \(\dd^7_{27}\), \(\dd^{27}_7\), \(\dd^{14}_{27}\), \(\dd^{27}_{14}\) and \(\dd^{27}_{27}\).  Explicitly:

\begin{itemize}
\item For \(f\in\Om^0(\M)\):
\e\label{1KI}
\dd f = \dd^1_7 f, \hs{5mm} \dd(f\ph) = \dd^1_7 f\w \ph \et \dd(f\ps) = \dd^1_7 f\w\ps;
\ee
\item For \(\al\in\Om^1(\M)\):\footnote{N.B.\ the formula for \(\dd(\al \w \ps)\) is incorrectly stated in \cite[\S5]{RoG2S} as \(-\Hs_\ph \dd^7_7\).  The error was pointed out by Bryant--Xu in \cite[\S1.3]{LFfCG2S:STB}.}
\ew
\begin{gathered}
\dd\al = \frac{1}{3}\Hs_\ph\lt(\dd^7_7\al\w\Hs_\ph\ph\rt) + \dd^7_{14}\al, \hs{5mm} \dd(\al\w\ph) = \frac{2}{3}\dd^7_7\al\w\ps - \Hs_\ph\dd^7_{14}\al,\\
\dd\Hs_\ph(\al\w\ph) = \frac{4}{7}(\dd^7_1\al)\ps + \frac{1}{2}\dd^7_7\al\w\ph + \Hs_\ph\dd^7_{27}\al,\\
\dd\lt(\Hs_\ph(\al\w\Hs_\ph\ph)\rt) = -\frac{3}{7}\lt(\dd^7_1\al\rt)\ph - \frac{1}{2}\Hs_\ph\lt(\dd^7_7\al\w\ph\rt) + \dd^7_{27}\al,\\
\dd(\al \w \ps) = \Hs_\ph \dd^7_7 \al \et \dd(\Hs_\ph\al) = -\lt(\dd^7_1\al\rt) vol_\ph.
\end{gathered}
\eew
\item For \(\be\in\Om^2_{14}(\M)\):
\e\label{14KI}
\dd\be = \frac{1}{4}\Hs_\ph(\dd^{14}_7\be\w\ph) + \dd^{14}_{27}\be \et \dd^*\be = \dd^{14}_7\be;
\ee
\item For \(\ga\in\Om^3_{27}(\M)\):
\e\label{27KI}
\begin{gathered}
\dd\ga = \frac{1}{4}\dd^{27}_7\ga \w\ph + \Hs_\ph\dd^{27}_{27}\ga \et \dd^*\ga = \frac{1}{3}\Hs_\ph(\dd^{27}_7\ga\w\ps) + \Hs_\ph\dd^{27}_{14}\ga.
\end{gathered}
\ee
\end{itemize}
The condition \(\dd^2 = 0\) corresponds to the 14 identities:
\e\label{d2=0}
\begin{gathered}
\dd^7_7\dd^1_7 = 0, \hs{5mm} \dd^7_{14}\dd^1_7 = 0, \hs{5mm} \dd^7_1\dd^7_7 = 0, \hs{5mm} \dd^{14}_7 \dd^7_{14} = \frac{2}{3}\lt(\dd^7_7\rt)^2, \hs{5mm} \dd^{27}_7\dd^7_{27} = \lt(\dd^7_7\rt)^2 + \frac{12}{7}\dd^1_7\dd^7_1,\\
\\
\dd^7_{14}\dd^7_7 + 2\dd^{27}_{14}\dd^7_{27} = 0, \hs{5mm} 3\dd^{14}_{27}\dd^7_{14} + \dd^7_{27}\dd^7_7 = 0, \hs{5mm} 2\dd^{27}_{27}\dd^7_{27} - \dd^7_{27}\dd^7_7 = 0, \hs{5mm} \dd^7_1 \dd^{14}_7 = 0,\\
\\
\dd^7_7\dd^{14}_7 + 2\dd^{27}_7\dd^{14}_{27} = 0, \hs{5mm} \dd^7_{27}\dd^{14}_7 + 4\dd^{27}_{27}\dd^{14}_{27} = 0, \hs{5mm} 3\dd^{14}_7\dd^{27}_{14} + \dd^7_7\dd^{27}_7 = 0,\\
\\
2\dd^{27}_7\dd^{27}_{27} - \dd^7_7\dd^{27}_7 = 0 \et \dd^7_{14}\dd^{27}_7 + 4\dd^{27}_{14}\dd^{27}_{27} = 0.
\end{gathered}
\ee

Finally, all Hodge Laplacians can be expressed in terms of the same operators.  Explicitly:
\begin{itemize}
\item For \(f\in\Om^0(\M)\):
\ew
\De f = \dd^7_1\dd^1_7 f.
\eew

\item For \(\al\in\Om^1(\M)\):
\ew
\De\al = (\dd^7_7)^2\al + \dd^1_7\dd^7_1\al.
\eew

\item For \(\be\in\Om^2_{14}(\M)\):
\ew
\De\be = \frac{5}{4}\dd^7_{14}\dd^{14}_7\be + \dd^{27}_{14}\dd^{14}_{27}\be.
\eew

\item For \(\ga\in\Om^3_{27}(\M)\):
\e\label{27-De}
\De\ga = \frac{7}{12}\dd^7_{27}\dd^{27}_7\ga + \dd^{14}_{27}\dd^{27}_{14}\ga + (\dd^{27}_{27})^2\ga.
\ee
\end{itemize}
Formulae for the Hodge Laplacian acting on sections of the remaining bundles \(\ww[q]{p}\T^*\M\) are obtained by identifying \(\ww[q]{p}\T^*\M\) with either \(\ww{0}\T^*\M\), \(\ww{1}\T^*\M\), \(\ww[14]{2}\T^*\M\) or \(\ww[27]{3}\T^*\M\) as appropriate, and noting that, since \(\ph\) is torsion-free, \(\De\) commutes with the identification (so that, e.g.\ \(\De(f\ph) = (\De f)\ph\)).\\
\end{Thm*}

~\vs{5mm}

\noindent Laurence H.\ Mayther\\
University of Cambridge\\
United Kingdom\\
{\it lhm32@cam.ac.uk}

\end{document}

%% file: Festino_1pt0.tex
\usepackage{adjustbox,array,bigints,cancel,color,comment,extpfeil,mathdots,mathrsfs,mathtools,MnSymbol,multicol,scalerel,setspace,tcolorbox,tikz,tikz-cd,upgreek}
\usepackage[fontsize = 10pt]{fontsize}
\usepackage[left = 2.5cm, right = 2.5cm, top = 2.5cm, bottom = 2.5cm]{geometry}
\usepackage{hyperref}
\usetikzlibrary{patterns}
\numberwithin{equation}{section}

\theoremstyle{plain}
\newtheorem{Cl}[equation]{Claim}

\newtheorem{Lem}[equation]{Lemma}

\newtheorem{Prop}[equation]{Proposition}

\newtheorem{Thm}[equation]{Theorem}

\theoremstyle{definition}

\newtheorem{Defn}[equation]{Definition}
\newtheorem{Ex}[equation]{Example}

\theoremstyle{remark}
\newtheorem{Rk}[equation]{Remark}

\theoremstyle{plain}
\newtheorem*{Cl*}{Claim}
\newtheorem*{Conj*}{Conjecture}
\newtheorem*{Lem*}{Lemma}
\newtheorem*{Prop*}{Proposition}
\newtheorem*{Q*}{Question}
\newtheorem*{Schol*}{Scholium}
\newtheorem*{SubCl*}{Subclaim}
\newtheorem*{Thm*}{Theorem}

\theoremstyle{definition}
\newtheorem*{Cond*}{Condition}
\newtheorem*{Cstr*}{Construction}
\newtheorem*{Defn*}{Definition}
\newtheorem*{Ex*}{Example}
\newtheorem*{Exs*}{Examples}
\newtheorem*{Md*}{Method}
\newtheorem*{Nt*}{Notation}
\newtheorem*{Pty*}{Property}

\theoremstyle{remark}

\newtheorem*{Rk*}{Remark}
\newtheorem*{Rks*}{Remarks}
\newtheorem*{A-d}{Aside}

\newcommand{\cag}{\begin{equation}\begin{gathered}}
\newcommand{\caag}{\end{gathered}\end{equation}}
\newcommand{\caw}{\begin{equation*}\begin{gathered}}
\newcommand{\caaw}{\end{gathered}\end{equation*}}
\newcommand{\e}{\begin{equation}\begin{aligned}}
\newcommand{\ee}{\end{aligned}\end{equation}}
\newcommand{\ew}{\begin{equation*}\begin{aligned}}
\newcommand{\eew}{\end{aligned}\end{equation*}}

\newcommand{\bcd}{\begin{tikzcd}}
\newcommand{\ecd}{\end{tikzcd}}
\newcommand{\bma}{\begin{matrix}}
\newcommand{\ema}{\end{matrix}}
\newcommand{\bpm}{\begin{pmatrix}}
\newcommand{\epm}{\end{pmatrix}}
\newcommand{\bvm}{\begin{vmatrix}}
\newcommand{\evm}{\end{vmatrix}}

\newcommand{\nts}{\begin{tcolorbox}}
\newcommand{\ntss}{\end{tcolorbox}}

\newcommand{\aref}[1]{Appendix \ref{#1}}

\newcommand{\cref}[1]{Corollary \ref{#1}}

\newcommand{\dref}[1]{Definition \ref{#1}}

\newcommand{\eref}[1]{eqn.\hspace{0.6mm}(\ref{#1})}
\newcommand{\erefs}[1]{eqns.\hspace{0.6mm}(\ref{#1})}

\newcommand{\lref}[1]{Lemma \ref{#1}}

\newcommand{\pref}[1]{Proposition \ref{#1}}

\newcommand{\sref}[1]{\S\ref{#1}}
\newcommand{\srefs}[1]{\S\S\ref{#1}}
\newcommand{\tref}[1]{Theorem \ref{#1}}
\newcommand{\trefs}[1]{Theorems \ref{#1}}

\DeclareMathSizes{10}{10}{8}{7}

\newcommand{\mss}[1]{\mbox{\scriptsize \(#1\)}}
\newcommand{\mfn}[1]{\mbox{\footnotesize \(#1\)}}

\newcommand{\mns}[1]{\mbox{\normalsize \(#1\)}}
\newcommand{\mla}[1]{\mbox{\large \(#1\)}}

\newcommand{\bb}[1]{\mathbb{#1}}
\newcommand{\cal}[1]{\mathscr{#1}}
\newcommand{\fr}[1]{\mathfrak{#1}}

\newcommand{\mc}[1]{\mathcal{#1}}

\newcommand{\et}{\hspace{5mm}\text{and}\hspace{5mm}}
\newcommand{\hs}[1]{\hspace{#1}}

\newcommand{\vs}[1]{\vspace{#1}}

\DeclareMathSymbol{\Alpha}{\mathalpha}{operators}{"41}
\DeclareMathSymbol{\Beta}{\mathalpha}{operators}{"42}
\DeclareMathSymbol{\Epsilon}{\mathalpha}{operators}{"45}
\DeclareMathSymbol{\Zeta}{\mathalpha}{operators}{"5A}
\DeclareMathSymbol{\Eta}{\mathalpha}{operators}{"48}
\DeclareMathSymbol{\Iota}{\mathalpha}{operators}{"49}
\DeclareMathSymbol{\Kappa}{\mathalpha}{operators}{"4B}
\DeclareMathSymbol{\Mu}{\mathalpha}{operators}{"4D}
\DeclareMathSymbol{\Nu}{\mathalpha}{operators}{"4E}
\DeclareMathSymbol{\Omicron}{\mathalpha}{operators}{"4F}
\DeclareMathSymbol{\Rho}{\mathalpha}{operators}{"50}
\DeclareMathSymbol{\Tau}{\mathalpha}{operators}{"54}
\DeclareMathSymbol{\Chi}{\mathalpha}{operators}{"58}
\DeclareMathSymbol{\omicron}{\mathord}{letters}{"6F}

\newcommand{\al}{\alpha}
\newcommand{\be}{\beta}
\newcommand{\ga}{\gamma}
\newcommand{\de}{\delta}

\newcommand{\ze}{\zeta}
\renewcommand{\th}{\theta}
\newcommand{\vth}{\vartheta}
\newcommand{\io}{\iota}

\newcommand{\la}{\lambda}
\newcommand{\vpi}{\varpi}
\newcommand{\rh}{\rho}

\newcommand{\si}{\sigma}

\newcommand{\ph}{\phi}
\newcommand{\vph}{\upvarphi}

\newcommand{\ch}{\chi}
\newcommand{\ps}{\psi}
\newcommand{\om}{\omega}

\newcommand{\Ga}{\Gamma}
\newcommand{\De}{\Delta}

\newcommand{\Ps}{\Psi}
\newcommand{\Om}{\Omega}

\newcommand{\<}{\langle}
\newcommand{\?}{\rangle}

\newcommand{\bqt}[3]{\left.\raisebox{-1mm}{\(#2\)}\middle\backslash\raisebox{1mm}{\(#1\)}\middle/\raisebox{-1mm}{\(#3\)}\right.}

\newcommand{\Ds}{\bigoplus}
\newcommand{\ds}{\oplus}
\newcommand{\End}{\operatorname{End}}

\newcommand{\Id}{\operatorname{Id}}
\newcommand{\lqt}[2]{\left.\raisebox{-1mm}{\(#2\)}\middle\backslash\raisebox{1mm}{\(#1\)}\right.}

\newcommand{\proj}{\operatorname{proj}}

\newcommand{\rqt}[2]{\left.\raisebox{1mm}{\(#1\)}\middle/\raisebox{-1mm}{\(#2\)}\right.}
\newcommand{\ts}{\otimes}

\newcommand{\Tr}{\operatorname{Tr}}

\newcommand{\x}{\times}

\newcommand{\Hocl}[1]{\overset{\circ}{H^{#1}}_{\kern-1.9mm\cl}}

\newcommand{\lop}{\left\|\kern-1.30mm\left\|}
\newcommand{\op}{\|\kern-1.30mm\|}
\newcommand{\rop}{\right\|\kern-1.30mm\right\|}
\newcommand{\SI}{\operatorname{\cal{I}\kern-1.5pt nd}}
\newcommand{\Spec}{\operatorname{Spec}}

\newcommand{\cc}{\subseteq}

\newcommand{\es}{\emptyset}

\newcommand{\LE}{\Leftrightarrow}

\newcommand{\mt}{\mapsto}

\newcommand{\osr}{\backslash}
\newcommand{\oto}[1]{\xrightarrow{#1}}
\newcommand{\pc}{\subset}

\newcommand{\CCH}{\mathcal{H}_4}

\newcommand{\CH}{\mathcal{H}_3}

\newcommand{\cl}{\mathrm{closed}}

\newcommand{\dd}{\mathrm{d}}

\newcommand{\Diff}{\operatorname{Diff}}
\newcommand{\dR}[1]{H^{#1}_{\operatorname{dR}}}

\newcommand{\emb}{\hookrightarrow}

\newcommand{\hk}{\righthalfcup}

\newcommand{\Hs}{\raisebox{1pt}{\mss{\bigstar}}}

\newcommand{\ls}{\left\lsem}

\newcommand{\M}{\mathrm{M}}

\newcommand{\N}{\mathrm{N}}

\newcommand{\rs}{\right\rsem}

\renewcommand{\ss}[2][{}]{\bigodot{\hspace{-1mm}}^{#2}_{#1}\hspace{0.6mm}}

\newcommand{\T}{\mathrm{T}}
\newcommand{\tl}{{\mns{\sim}}}

\newcommand{\w}{\wedge}
\newcommand{\ww}[2][{}]{\bigwedge{\hspace{-1mm}}^{#2}_{\hspace{1mm}#1}\hspace{0.1mm}}

\newcommand{\0}{\infty}
\newcommand{\1}{\cdot}

\renewcommand{\ge}{\geqslant}
\newcommand{\gl}{\hspace{0.4mm}\raisebox{0.8mm}{\(>\)}\kern-1.8mm\raisebox{-0.8mm}{\(<\)}\hspace{0.4mm}}
\newcommand{\gle}{\hspace{0.4mm}\raisebox{1.2mm}{\(\ge\)}\kern-1.8mm\raisebox{-1.2mm}{\(\le\)}\hspace{0.4mm}}
\renewcommand{\le}{\leqslant}
\newcommand{\pt}{\bullet}
\newcommand{\sgn}{\operatorname{sign}}

\newcommand{\g}{\(\mathrm{G}_2\)}
\newcommand{\Gg}{\mathrm{G}}
\newcommand{\GL}{\operatorname{GL}}
\newcommand{\Norm}{\operatorname{Norm}}

\newcommand{\sdp}{\ltimes}
\newcommand{\sg}{\(\widetilde{\mathrm{G}}_2\)}
\newcommand{\SL}{\operatorname{SL}}

\newcommand{\SO}{\operatorname{SO}}

\newcommand{\Stab}{\operatorname{Stab}}
\newcommand{\SU}{\operatorname{SU}}

\let\LE\iff%Saves original definition of '\iff' before redefining.
\renewcommand{\iff}{if and only if}

\newcommand{\stp}{suffices to prove}

\newcommand{\wrt}{with respect to}

\newcommand{\CGN}{Crowley--Goette--Nordstr\"om}

\newcommand{\ol}[1]{\overline{#1}}

\newcommand{\lt}{\left}
\newcommand{\m}{\middle}

\newcommand{\rt}{\right}

\newcommand{\tld}{\widetilde}

%% file: Spectral_invariants_of_Joyce_orbifolds.bbl
\begin{thebibliography}{10}
\bibitem{O&ST} Adem, A., Leida, J. and Ruan, Y., {\it Orbifolds and Stringy Topology}, Cambridge Tracts in Mathematics, {\bf 171} (Cambridge University Press, New York (NY), 2007).

\bibitem{OtHE&tIT} Atiyah, M.F., Bott, R. and Patodi, V.K., `On the heat equation and the index theorem', {\it Invent. Math.} {\bf 19} (1973), 279--330.

\bibitem{APSI} Atiyah, M.F., Patodi, V.K. and Singer, I.M., `Spectral asymmetry and Riemannian geometry. I', {\it Math. Proc. Camb. Philos. Soc.} {\bf 77} (1975), 43--69.

\bibitem{APSII} Atiyah, M.F., Patodi, V.K. and Singer, I.M., `Spectral asymmetry and Riemannian geometry. II', {\it Math. Proc. Camb. Philos. Soc.} {\bf 78} (1975), 405--432.

\bibitem{APSIII} Atiyah, M.F., Patodi, V.K. and Singer, I.M., `Spectral asymmetry and Riemannian geometry. III', {\it Math. Proc. Camb. Philos. Soc.} {\bf 79} (1976), 71--99.

\bibitem{TDTfVM} Baily Jr., W.L., `The decomposition theorem for \(V\)-manifolds', {\it Am. J. Math.} {\bf 78} (1956), no. 4, 862--888.

\bibitem{SoGH-LOoFT&RS} Beitz, S.F., `Spectrum of generalised Hodge-Laplace operator on flat tori and round spheres', {\it Osaka J. Math.} {\bf 57} (2020), 357--379.

\bibitem{RoG2S} Bryant, R.L., `Some remarks on \g-structures', in {\it Proceedings of the G\"{o}kova Geometry-Topology Conference 2005}, ed. by S. Akbulut, T. \"{O}nder and R.J. Stern, (International Press, Somerville (MA), 2006), 75--109.

\bibitem{LFfCG2S:STB} Bryant, R.L. and Xu, F., `Laplacian flow for closed \g-structures: short time behavior', arXiv:1101.2004 [math.DG] (2011).

\bibitem{ACCY3FfWF3F} Corti, A., Haskins, M., Nordstr\"{o}m, J. and Pacini, T., `Asymptotically cylindrical Calabi-Yau 3-folds from weak Fano 3-folds', {\it Geom. Topol.} {\bf 17} (2013), no. 4, 1955--2059.

\bibitem{AAIoG2M} Crowley, D., Goette, S. and Nordström, J., `An analytic invariant of \g\ manifolds', arXiv:1505.02734 [math.GT] (2015).

\bibitem{AEBVPfG2S} Donaldson, S.K., `An elliptic boundary value problem for \g-structures', {\it Ann. Inst. Fourier, Grenoble} {\bf 68} (2018), no. 7, 2783--2809.

\bibitem{THKLFPFftScDO} Duistermaat, J.J., {\it The Heat Kernel Lefschetz Fixed Point Formula for the Spin-c Dirac Operator}, Progress in Nonlinear Differential Equations and Their Applications, {\bf 18} (Birkh\"{a}user, Boston (MA), 1996).

\bibitem{ZTAZI} Epstein, P., `Zur Theorie allgemeiner Zetafunctionen. I', {\it Math. Ann.} {\bf 56} (1903), no. 4, 615--644.

\bibitem{OeI} Farsi, C., `Orbifold \(\eta\)-invariants', {\it Indiana Univ. Math. J.} {\bf 56} (2007), no. 2, 501--521.

\bibitem{RTAFC} Fulton, W. and Harris, J., {\it Representation Theory: A First Course}, Graduate Texts in Mathematics, {\bf 129} (Springer, New York (NY), 1991).

\bibitem{nIoETCS} Goette, S., Nordstr\"{o}m, J. and Zagier, D., `\(\nu\)-invariants of extra-twisted connected sums', arXiv:2010.16367 [math.GT] (2020).

\bibitem{STBoaMLCoG2S} Grigorian, S., `Short-time behaviour of a modified Laplacian coflow of \g-structures', {\it Adv. Math.} {\bf 248} (2013), 378--415.

\bibitem{A3F&ESLGoTG2} Herz, C., `Alternating 3-forms and exceptional simple Lie groups of type \g', {\it Can. J. Math.} {\bf XXXV} (1983), no. 5, 776--806.

\bibitem{TGo3Fi6&7D} Hitchin, N.J., `The geometry of three-forms in six and seven dimensions', arXiv:math/0010054 [math.DG] (2000).

\bibitem{SF&SM} Hitchin, N.J., `Stable forms and special metrics', in {\it Global Differential Geometry: The Mathematical Legacy of Alfred Gray}, ed. M. Fern\'{a}ndez and J.A. Wolf, Contemporary Mathematics, {\bf 288} (American Mathematical Society, Providence (RI), 2001), 70--89.

\bibitem{CR7MwHG2I} Joyce, D.D., `Compact Riemannian 7-manifolds with holonomy \g. I', {\it J. Differ. Geom.} {\bf 43} (1996), no. 2, 291--328.

\bibitem{CR7MwHG2II} Joyce, D.D., `Compact Riemannian 7-manifolds with holonomy \g. II', {\it J. Differ. Geom.} {\bf 43} (1996), no. 2, 329--375.

\bibitem{CMwSH} Joyce, D.D., {\it Compact Manifolds with Special Holonomy}, Oxford Mathematical Monographs (Oxford University Press, New York (NY), 2000).

\bibitem{TSTfVM} Kawasaki, T., `The signature theorem for \(V\)-manifolds', {\it Topology} {\bf 17} (1978), 75--83.

\bibitem{TCS&SRH} Kovalev, A.G., `Twisted connected sums and special Riemannian holonomy', {\it J. Reine Angew. Math.} {\bf 565} (2003), 125--160.

\bibitem{UA&BotDHFoG2&SG2F} Mayther, L.H., `Unboundedness above and below of the Donaldson--Hitchin functionals on \g- and \sg-forms', \href{https://arxiv.org/abs/2308.15438}{arXiv:2308.15438 [math.DG]} (2023).

\bibitem{MT} Milnor, J.W., {\it Morse Theory}, Annals of Mathematics Studies, {\bf 51} (Princeton University Press, Princeton (NJ), 1963).

\bibitem{ETCSG2M} Nordstr\"{o}m, J., `Extra-twisted connected sum \g-manifolds', {\it Ann. Glob. Anal. Geom.} {\bf 64} (2023), no. 1, article no. 2.

\bibitem{DVdPDeGKG} Peter, F. and Weyl, H., `Die Vollst\"{a}ndigkeit der primitiven Darstellungen einer geschlossenen kontinuierlichen Gruppe', {\it Math. Ann.} {\bf 97} (1927), 737--755.

\bibitem{RG&HG} Salamon, S.M., {\it Riemannian Geometry and Holonomy Groups}, Pitman Research Notes in Mathematics Series, {\bf 201} (Longman Scientific and Technical, Essex, 1989).

\bibitem{CPoaEO} Seeley, R.T., `Complex powers of an elliptic operator', in {\it Singular Integrals}, ed. by A.P. Calder\'{o}n, Proceedings of Symposia in Pure Mathematics, {\bf X} (American Mathematical Society, Providence (RI), 1967), 288--307.

\bibitem{FoDM&LG} Warner, F.W., {\it Foundations of Differentiable Manifolds and Lie Groups}, (Scott, Foresman \& Company, 1971; repr. as Graduate Texts in Mathematics, {\bf 94}, Springer-Verlag, Berlin, 1983).

\bibitem{TESaMfFT} Wolpert, S., `The eigenvalue spectrum as moduli for flat tori', {\it Trans. Am. Math. Soc.} {\bf 244} (1978), 313--321.
\end{thebibliography}
